\RequirePackage{fix-cm}
\documentclass{svjour3}                     
\smartqed  
\usepackage{graphicx}
\usepackage{latexsym,amsfonts,amssymb,amsmath,amscd}
\usepackage{subfig} 
\usepackage{bbold}
\usepackage{dsfont}

\newcommand{\suchthat}{\;\ifnum\currentgrouptype=16 \middle\fi|\;}

\newcommand{\cov}{\mathrm{Cov}}
\newcommand{\var}{\mathrm{Var}}
\begin{document}

\title{Gaussian Process Emulators for Computer Experiments with Inequality Constraints
}


\author{Hassan Maatouk         \and
        Xavier Bay
}


\institute{H. Maatouk \at
              \'Ecole des Mines de St-\'Etienne, 158 Cours Fauriel, Saint-\'Etienne, France \\
              Tel.: +336-49-61-74-41\\
              \email{hassan.maatouk@mines-stetienne.fr}           
           \and
           X. Bay \at
              \'Ecole des Mines de St-\'Etienne, 158 Cours Fauriel, Saint-\'Etienne, France
}

\date{Received: date / Accepted: date}

\maketitle

\begin{abstract}
Physical phenomena are observed in many fields (science and engineering)
and are often studied by time-consuming computer codes. These codes are analyzed with statistical models, often called emulators. In many situations, the physical system (computer model output) may be known to satisfy inequality constraints with respect to some or all input variables. Our aim is to build a model capable of incorporating both data interpolation and inequality constraints into a Gaussian process emulator.  
By using a functional decomposition, we propose a finite-dimensional approximation of Gaussian processes such that all conditional simulations satisfy the inequality constraints in the entire domain. The inequality mean and mode (i.e. mean and maximum \textit{a posteriori}) of the conditional Gaussian process are calculated and prediction intervals are quantified. To show the performance of the proposed model, some conditional simulations with inequality constraints such as boundedness, monotonicity or convexity conditions in one and two dimensions are given. A simulation study to investigate the efficiency of the method in terms of prediction and uncertainty quantification is included.
\end{abstract}

\section{Introduction}
\label{intro}
In the engineering activity, runs of a computer code can be expensive and time-consuming. One solution is to use a statistical surrogate for conditioning
computer model outputs at some input locations (design points). Gaussian process
(GP) emulator is one of the most popular choices \cite{sacks1989}. The reason comes from the property of the GP that uncertainty can be quantified.
Furthermore, it has several nice properties. For example, the conditional GP at observation data (linear equality constraints) is still a GP \cite{cramer1967stationary}. Additionally, some inequality constraints (such as monotonicity and convexity) of output computer responses are related to partial derivatives. In such cases, the partial derivatives of the GP are also Gaussian Processes (GPs) (see e.g. \cite{cramer1967stationary} and \cite{opac-b1081425}). Incorporating an infinite number of linear inequality constraints into a GP emulator, the problem becomes more difficult. The reason is that the resulting conditional process is not a GP.\

In the literature of interpolation with inequality constraints, we find two types of methods. The first one is deterministic and based on splines, which have the advantage that inequality constraints are satisfied in the entire domain (see e.g. \cite{fritsch1980monotone}, \cite{doi10.1137/0909048}, \cite{1987}, \cite{Wolberg2002145} and \cite{Wright1980}).
The second one is based on the simulation of the conditional GP by using the subdivision of the input set (see e.g. \cite{Abrahamsen2001}, \cite{DaVeiga2012}, \cite{golchi2015monotone}, \cite{journals/jmlr/RiihimakiV10} and \cite{Wang2012}). In that case, the inequality constraints are satisfied in a finite number of input locations. However, uncertainty can be quantified. In this framework, \textit{constrained} Kriging has been studied in the domain of geostatistics (see e.g. \cite{Fouquet93} and \cite{kleijnen2012monotonicity}). In previous work, some methodologies have been based on the knowledge of the derivatives of the GP at some input locations (see e.g. \cite{golchi2015monotone}, \cite{journals/jmlr/RiihimakiV10} and \cite{Wang2012}). For monotonicity constraints with noisy data, a Bayesian approach was developed in \cite{journals/jmlr/RiihimakiV10}. In \cite{golchi2015monotone} the problem is to build a GP emulator by using the \textit{prior} monotonicity information of the computer model response with respect to some inputs. Their idea is based on an approach similar to \cite{journals/jmlr/RiihimakiV10} placing the derivatives information at specified input locations, by forcing the derivative process to be positive at these points. In such methodology, monotonicity constraints are not guaranteed in the entire domain. Recently, a methodology based on a discrete-location approximation for incorporating inequality constraints into a GP emulator was developed in \cite{DaVeiga2012}. Again, the inequality constraints are not guaranteed in the entire domain.\

On the other hand, Villalobos and Wahba \cite{1987} used splines to estimate an interpolation smooth function satisfying a finite number of linear inequality constraints. In term of estimation of monotone smoothing functions, using B-splines was firstly introduced by Ramsay \cite{ramsay1988}. The idea is based on the integration of B-splines defined on a properly set of knots with positive coefficients to ensure monotonicity constraints. A similar approach is applied to econometrics in \cite{Dole1999}. Xuming \cite{He96monotoneb-spline} takes the same approach and suggests the calculation of the coefficients by solving a finite linear minimization problem. A comparison to monotone kernel regression and an application to decreasing constraints are included.\

Our aim in this paper is to build a GP emulator incorporating the advantage of splines approach in order to ensure that inequality constraints are satisfied in the entire domain. We propose a finite-dimensional approximation of Gaussian Processes that converges uniformly pathwise. Usually, finite-dimensional approximations of Gaussian Processes (see e.g. \cite{trecate1999finite}) are truncated Karhunen-Lo\`eve decompositions, where the random coefficients are independent and the basis functions are the eigenfunctions of the covariance function describing the Gaussian process. It is not the case in this paper, the finite-dimensional model is also a linear decomposition of deterministic basis functions with Gaussian random coefficients but the coefficients are not independent. We show that the basis functions can be chosen such that inequality constraints of the GP are \textit{equivalent} to constraints on the coefficients. Therefore, the inequality constraints are reduced to a finite number of constraints. Furthermore, any posterior sample of coefficients leads to an interpolating function satisfying the inequality constraints in the entire domain.
Finally, the problem is reduced to simulate a Gaussian vector (random coefficients) restricted to convex sets which is a well-known problem with existing algorithms (see e.g. \cite{Botts}, \cite{Chopin2011FST19607241960748}, \cite{Emery2013}, \cite{Fouquet93},\cite{Geweke91efficientsimulation}, \cite{Maatouk2014}, \cite{journals/sac/PhilippeR03} and \cite{Robert}).\

The article is structured as follows~: in Sect.~\ref{GPCE}, we briefly recall Gaussian process modeling for computer experiments and the choice of covariance functions. In Sect.~\ref{GPMC}, we propose a finite-dimensional approximation of GPs capable of interpolating computer model outputs and incorporating inequality constraints in the entire domain, and we investigate its properties. In Sect.~\ref{SimStudy}, the performance of the proposed model in terms of prediction and uncertainty quantification using the simulation study in \cite{golchi2015monotone} is investigated. In Sect.~\ref{illustrative}, we show some simulated examples of the conditional GP with inequality constraints (such as boundedness, monotonicity or convexity conditions) in one and two dimensions. Additionally, two cases of truncated simulations are studied. We end up this paper by some concluding remarks and future work.

\section{Gaussian process emulators for computer experiments}\label{GPCE}
We consider the model $y=f(\boldsymbol{x})$, where the simulator response $y$ is assumed to be a deterministic real-valued function of the d-dimensional variable $\boldsymbol{x}=(x_1,\ldots,x_d)\in \mathbb{R}^d$. We suppose that the real function is continuous and evaluated at $n$ design points given by the rows of the $n\times d$ matrix $\boldsymbol{X}=\left(\boldsymbol{x}^{(1)},\ldots,\boldsymbol{x}^{(n)}\right)^\top$, where $\boldsymbol{x}^{(i)}\in \mathbb{R}^d, \ 1\leq i\leq n$. In practice, the evaluation of the function is expensive and must be considered highly
time-consuming. The solution is to estimate the unknown function $f$ by using 
a GP emulator also known as ``Kriging''.
In this framework, $y$ is viewed as a realization of a continuous GP, 
\begin{equation*}
Y(\boldsymbol{x}):=\eta(\boldsymbol{x})+Z(\boldsymbol{x}),
\end{equation*}
where the deterministic continuous function $\eta~: \ \boldsymbol{x}\in \mathbb{R}^d  \ \longrightarrow \ \eta (\boldsymbol{x})\in R$
is the mean and $Z$ is a zero-mean GP with continuous covariance function 
\begin{equation*}
K~: \ (\boldsymbol{u},\boldsymbol{v})\in\mathbb{R}^d\times \mathbb{R}^d \  \longrightarrow \ K(\boldsymbol{u},\boldsymbol{v})\in\mathbb{R}.
\end{equation*} 
Conditionally to the observation $\boldsymbol{y}=\left(y\left(x^{(1)}\right),\ldots,y\left(x^{(n)}\right)\right)^\top$
the process is still a GP~:
\begin{equation}\label{condGP}
Y(\boldsymbol{x}) \suchthat Y\left(\boldsymbol{X}\right)=\boldsymbol{y}\sim\mathcal{N}\left(\zeta(\boldsymbol{x}),\tau^2(\boldsymbol{x})\right),
\end{equation}
where
\begin{equation*}
\begin{array}{ll}
\zeta(\boldsymbol{x})=\eta(\boldsymbol{x})+\boldsymbol{k}(\boldsymbol{x})^\top \mathbb{K}^{-1}\left(\boldsymbol{y}-\boldsymbol{\mu}\right)\\
\tau^2(\boldsymbol{x})=K(\boldsymbol{x},\boldsymbol{x})-\boldsymbol{k}(\boldsymbol{x})^\top \mathbb{K}^{-1}\boldsymbol{k}(\boldsymbol{x})
\end{array}
\end{equation*}
and $\boldsymbol{\mu}=\eta(\boldsymbol{X})$ is the vector of trend values at the experimental design points,
$\mathbb{K}_{i,j}=K\left(\boldsymbol{x}^{(i)},\boldsymbol{x}^{(j)}\right), \ i,j=1,\ldots,n$ is the covariance matrix of $Y(\boldsymbol{X})$ and
$\boldsymbol{k}(\boldsymbol{x})=\left(K\left(\boldsymbol{x}, \boldsymbol{x}^{(i)}\right)\right)$ is the vector of covariance between $Y\left(\boldsymbol{x}\right)$ and $Y\left(\boldsymbol{X}\right)$. Additionally, the covariance function between any two inputs can be written as~:
\begin{equation*}
C(\boldsymbol{x},\boldsymbol{x}'):=\cov\left(Y(\boldsymbol{x}),Y(\boldsymbol{x}') \suchthat Y(\boldsymbol{X})=\boldsymbol{y}\right)=K(\boldsymbol{x},\boldsymbol{x}')-\boldsymbol{k}(\boldsymbol{x})^\top\mathbb{K}^{-1}\boldsymbol{k}(\boldsymbol{x}'),
\end{equation*}
where $C$ is the covariance function of the conditional GP.
The mean $\zeta(\boldsymbol{x})$ is called Simple Kriging (SK) mean prediction of $Y(\boldsymbol{x})$ based on the computer model outputs $Y\left(\boldsymbol{X}\right)=\boldsymbol{y}$, \cite{jones1998}.

\subsection{The choice of covariance function}
The choice of $K$ has crucial consequences specially in controlling the smoothness of the Kriging metamodel.
It must be chosen in the set of definite and positive kernels. Some popular kernels are the Gaussian kernel, Mat\'ern kernel (with parameter $\lambda = 3/2,5/2,\ldots$) and exponential kernel (Mat\'ern kernel with parameter $\lambda=1/2$). Notice that these kernels are placed in order of smoothness, the Gaussian kernel corresponding to $\mathcal{C}^{\infty}$ function\footnote{The space of functions that admit derivatives of all orders.} and the exponential kernel to continuous one (see \cite{Rasmussen2005GPM1162254} and Table~\ref{kernel}).
In the running examples of this paper, we will consider the Gaussian kernel defined by
\begin{equation*}
K(\boldsymbol{x},\boldsymbol{x}'):=\sigma^2\prod_{k=1}^d \exp\left(-\frac{\left(x_k-x'_k\right)^2}{2\theta_k^2}\right),
\end{equation*} 
for all $\boldsymbol{x}, \ \boldsymbol{x}'\in \mathbb{R}^d$, where $\sigma^2$ and $\boldsymbol\theta=(\theta_1,\ldots,\theta_d)$ are parameters.

\begin{table}[hptb]
\centering
\caption{Some popular kernel functions used in Kriging methods.}
\label{kernel}
\begin{tabular}{ccc}
\hline \hline
Name  &  Expression &  Class   \\   
\hline
Gaussian       & $\sigma^2\exp\left(-\frac{(x-x')^2}{2\theta^2}\right)$ 	& $\mathcal{C}^{\infty}$       \\
Mat\'ern 5/2   & $\sigma^2\left(1+\frac{\sqrt{5}\mid x-x'\mid}{\theta}+\frac{5(x-x')^2}{3\theta^2}\right)\exp\left(-\frac{\sqrt{5} | x-x'|}{\theta}\right)$
& $\mathcal{C}^2$      \\
Mat\'ern 3/2   &  $\sigma^2\left(1+\frac{\sqrt{3}\mid x-x'\mid}{\theta}\right)\exp\left(-\frac{\sqrt{3}|x-x'|}{\theta}\right)$   & $\mathcal{C}^1$      \\
Exponential    &  $\sigma^2\exp\left(-\frac{|x-x'|}{\theta}\right)$ & $\mathcal{C}^0$      \\   
\hline
\end{tabular}
\end{table}

\subsection{Derivatives of Gaussian processes}
In this paragraph, we assume that the paths of $Y(\boldsymbol{x})$ are of class $\mathcal{C}^p$ (i.e. the space of functions that admit derivatives up to order $p$). This can be guaranteed if $K$ is smooth enough, and in particular if $K$ is of class $\mathcal{C}^{\infty}$ (see \cite{cramer1967stationary}).
The linearity of the differentiation operation ensures that the order partial
derivatives of a GP are also GPs \cite{cramer1967stationary}, with (see e.g. \cite{opac-b1081425})~: 
\begin{eqnarray*}
\mathds{E}\left(\partial_{x_k}^pY(\boldsymbol{x})\right)&=&\frac{\partial^p}{\partial x_k^p}\mathds{E}\left(Y(\boldsymbol{x})\right),\\
\cov\left(\partial_{x_k}^pY\left(\boldsymbol{x}^{(i)}\right),\partial_{x_{\ell}}^q Y\left(\boldsymbol{x}^{(j)}\right)\right)&=&\frac{\partial^{p+q}}{\partial x_k^p\partial (x'_{\ell})^q}K\left(\boldsymbol{x}^{(i)},\boldsymbol{x}^{(j)}\right).
\end{eqnarray*}

\section{Gaussian process emulators with inequality constraints} \label{GPMC}
In this section, we assume that the real function (physical system) may be known to satisfy inequality constraints (such as boundedness, monotonicity or convexity conditions) in the entire domain. Our aim is to incorporate both interpolation conditions and  inequality constraints into a Gaussian process emulator.

\subsection{Formulation of the problem}
Without loss of generality, the input $\boldsymbol{x}$ is in $[0,1]^d\subset\mathbb{R}^d$. We assume that the real function $f$ is evaluated at $n$ distinct locations in the input set,
\begin{equation*}
f\left(\boldsymbol{x}^{(i)}\right)=y_i, \qquad i=1,\ldots,n.
\end{equation*}
Let $(Y({\boldsymbol{x}}))_{\boldsymbol{x}\in [0,1]^d}$ be a zero-mean GP with covariance function $K$ and $\mathcal{C}^0\left([0,1]^d\right)$ the space of continuous function on $[0,1]^d$. We denote by $C$ the subset of $\mathcal{C}^0\left([0,1]^d\right)$ corresponding to a given set of linear inequality constraints. We aim to get the conditional distribution of $Y$ given interpolation conditions and inequality constraints respectively as
\begin{equation*}
\begin{array}{ll}
Y\left(\boldsymbol{x}^{(i)}\right)=y_i, \qquad i=1,\ldots,n,\\
Y\in C. 
\end{array}
\end{equation*}

\subsection{Gaussian process approximation}
To handle the conditional distribution incorporating both interpolation conditions and inequality constraints, we propose a finite-dimensional approximation of Gaussian processes of the form~:
\begin{equation}\label{proposedmodel}
Y^N(\boldsymbol{x}):=\sum_{j=0}^N\xi_j\phi_j(\boldsymbol{x}), \qquad \boldsymbol{x}\in\mathbb{R}^d,
\end{equation}
where $\boldsymbol{\xi}=(\xi_0,\ldots,\xi_N)^\top$ is a zero-mean Gaussian vector with covariance matrix $\Gamma^N$ and $\phi=(\phi_0,\ldots,\phi_N)^\top$ is a vector of basis functions. The choice of these basis functions and $\Gamma^N$ depend on the type of inequality constraints. Notice that $Y^N$ is a zero-mean GP with covariance function 
\begin{equation*}
K_N(\boldsymbol{x},\boldsymbol{x}')=\phi(\boldsymbol{x})^\top \Gamma^N\phi(\boldsymbol{x}').
\end{equation*}
The advantage of the proposed model \eqref{proposedmodel} is that the simulation of the conditional GP is reduced to the simulation of the Gaussian vector $\boldsymbol{\xi}$ given that 
\begin{eqnarray}\label{Gausconst1}
&&\sum_{j=0}^N\xi_j\phi_j\left(\boldsymbol{x}^{(i)}\right)=y_i, \qquad i=1,\ldots,n,\\ \label{Gausconst2}
&&\boldsymbol{\xi}\in C_{\boldsymbol{\xi}},
\end{eqnarray}
where $C_{\boldsymbol{\xi}}=\left\{\boldsymbol{c}\in \mathbb{R}^{N+1}~: \ \sum_{j=0}^Nc_j\phi_j\in C, \ \boldsymbol{c}=(c_0,\ldots,c_N)^\top\right\}$. Hence the problem is \textit{equivalent} to simulate a Gaussian vector restricted to \eqref{Gausconst1} and \eqref{Gausconst2}. In the following sections, we give some examples of the choice of the basis functions and we explain how we compute the covariance matrix $\Gamma^N$ of the Gaussian vector $\boldsymbol{\xi}$ to ensure the convergence of the finite-dimensional approximation $Y^N$ to the original GP $Y$.\

Note that the finite-dimensional model \eqref{proposedmodel} does not correspond to a truncated Karhunen-Lo\`eve expansion $Y(x)=\sum_{j=0}^{+\infty}Z_je_j(x)$ (see e.g. \cite{Rasmussen2005GPM1162254}) since the coefficients $\xi_j$ are not independent (unlike the coefficients $Z_j$) and the basis functions $\phi_j$ are not the eigenfunctions $e_j$ of the Mercer kernel $K(x,x')$.

\subsection{One dimensional cases}
\subsubsection{Boundedness constraints}\label{BPCOD}
We assume that the real function defined in the unit interval is continuous and respects boundedness constraints (i.e. $a\leq f(x) \leq b, \ x\in [0,1]$), where $-\infty\leq a<b\leq +\infty$. In that case, the convex set $C$ is the space of bounded functions and is defined as
\begin{equation*}
C:=\left\{f\in \mathcal{C}^0\left([0,1]\right)~: \ a\leq f(x)\leq b, \ x\in [0,1]\right\}.
\end{equation*}
Let us begin by constructing the functions $h_j, \ j=0,\ldots,N$ that will be used in the proposed model. We first descretize the input set as $0=u_0 < u_1 <\ldots < u_N=1$, and on each knot we build a function. For the sake of simplicity, we use a uniform subdivision of the input set, but the methodology can be adapted for any subdivision. For example at the $j^{\text{th}}$ knot $u_j=j\Delta_N=j/N$, the associated function is 
\begin{equation}\label{hj}  
h_j(x)=h\left(\frac{x-u_j}{\Delta_N}\right), \qquad j=0,\ldots, N,
\end{equation}
where $\Delta_N=1/N$ and $h(x):=\left(1-| x|\right)\mathbb{1}_{(| x|\leq 1)}, \ x\in\mathbb{R}$, see Figures~\ref{B-splines} and \ref{h} below for $N=4$. 
Notice that the $h_j$'s are bounded between $0$ and $1$
and $\sum_{j=0}^Nh_j(x)=1$ for all $x$ in $[0,1]$.
Additionally, the value of these functions at any knot $u_i, \ i=0,\ldots,N$ is equal to the Kronecker's delta ($h_j(u_i)=\delta_{ij}$), where $\delta_{ij}$ is equal to one if $i=j$
and zero otherwise. \\

\begin{figure}[hptb]
\begin{minipage}{.5\linewidth}
\centering
\subfloat[]{\label{B-splines}\includegraphics[scale=.3]{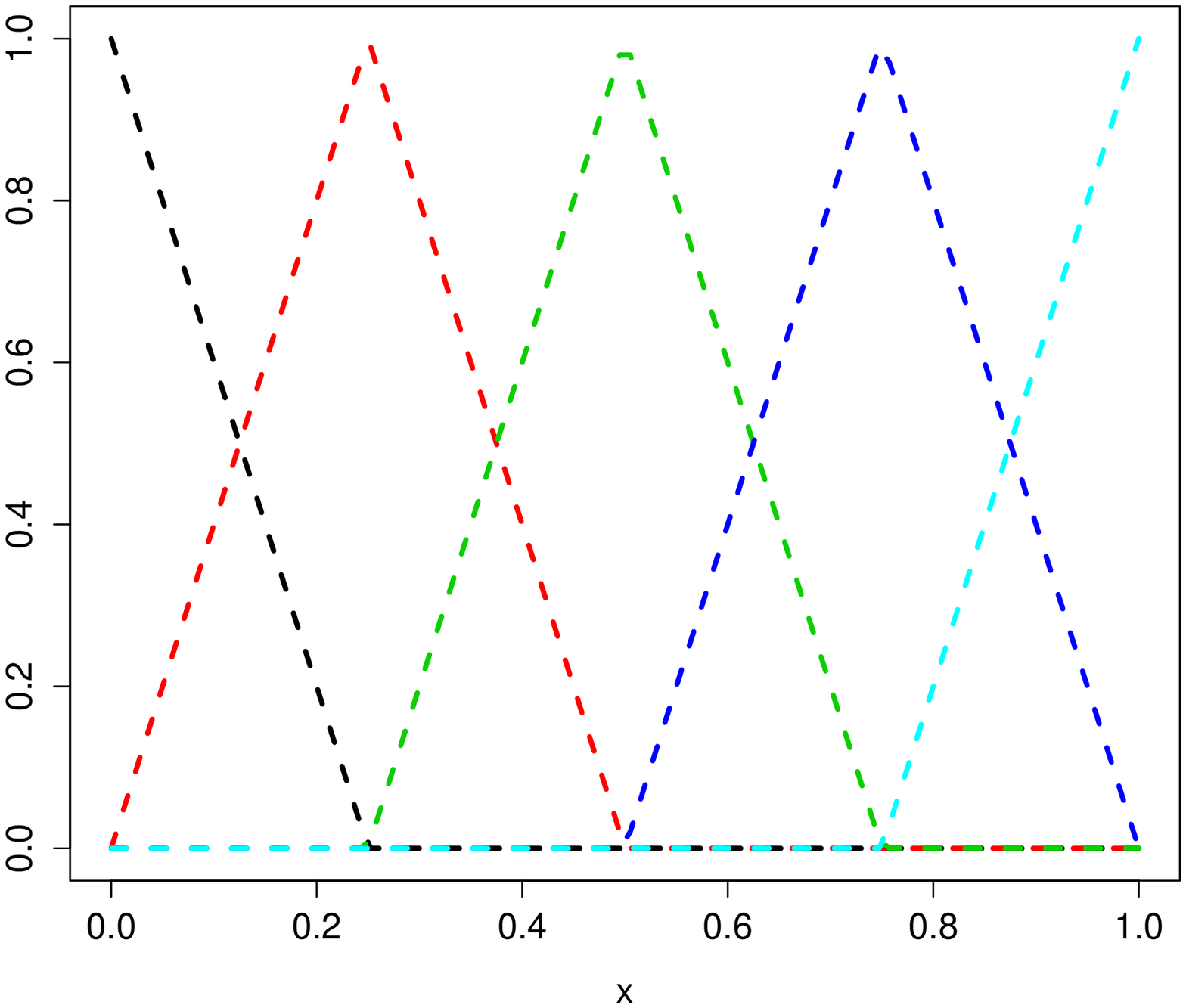}}
\end{minipage}%
\begin{minipage}{.5\linewidth}
\centering
\subfloat[]{\label{h}\includegraphics[scale=.3]{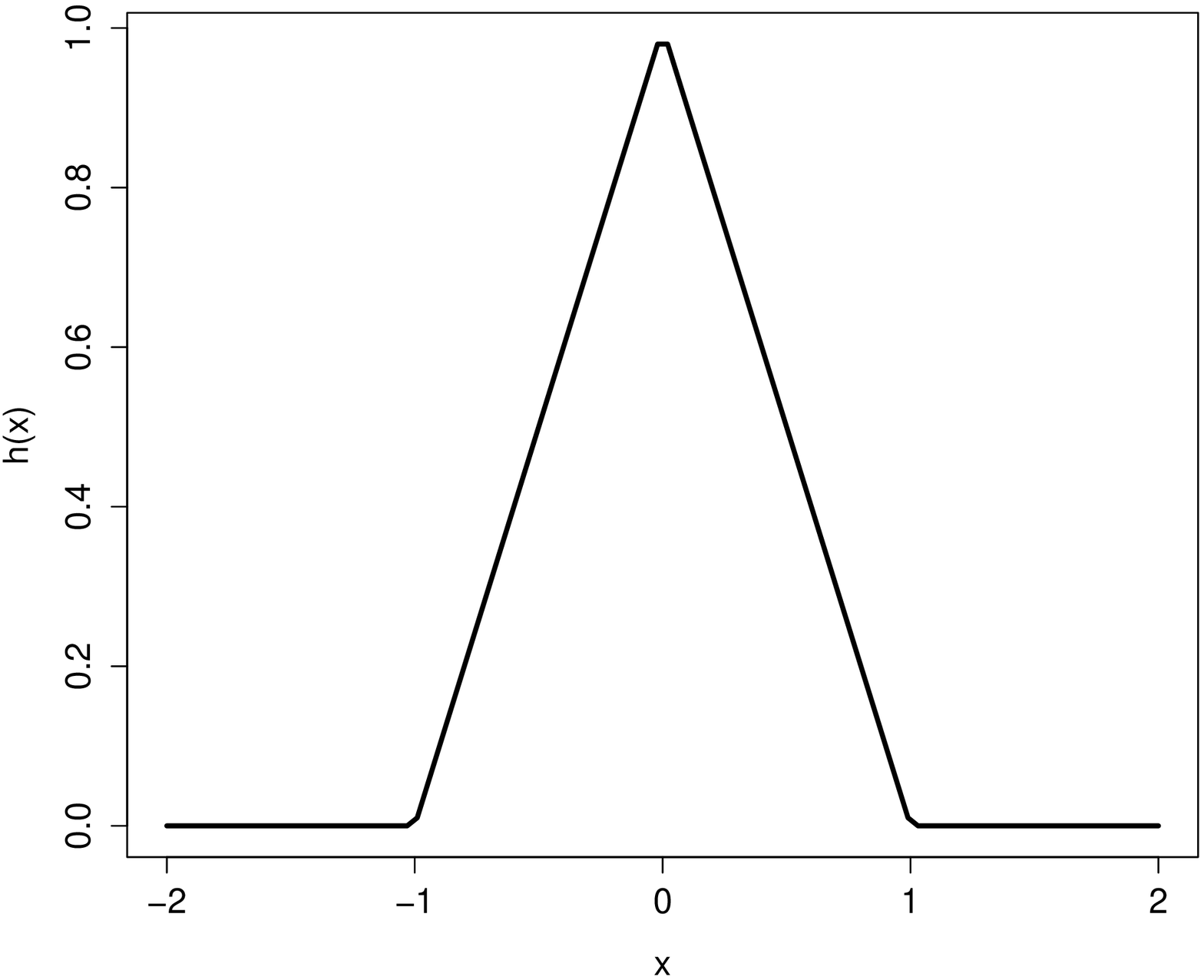}}
\end{minipage}
\caption{The basis functions $h_j, \ 0\leq j \leq 4$ (Figure~\ref{B-splines}) and the function $h$ (Figure~\ref{h}).}
\label{hih}
\end{figure}

\noindent The philosophy of the proposed method is presented in the following proposition~: \\

\begin{proposition}\label{boundaryproposition}
With the notations introduced before, the finite-dimensional approximation of GPs $(Y^N(x))_{x\in [0,1]}$ is defined as
\begin{equation}\label{boundaryapproach}
Y^N(x):=\sum_{j=0}^NY(u_j)h_j(x)=\sum_{j=0}^N\xi_jh_j(x),
\end{equation}
where $\xi_j=Y(u_j), \ j=0,\cdots,N$. If the realizations of the original GP $Y$ are continuous, then we have the following properties~: \\

\begin{itemize}
\item $Y^N$ is a finite-dimensional GP with covariance function $K_N(x,x')=h(x)^\top\Gamma^Nh(x')$, where $h(x)=(h_0(x),\ldots,h_N(x))^\top$, $\Gamma^N_{i,j}=K(u_i,u_j), \ i,j=0,\ldots,N$ and $K$ the covariance function of the original GP $Y$. \\

\item $Y^N$ converges uniformly pathwise to $Y$ when $N$ tends to infinity (with probability 1). \\

\item $Y^N$ is in $C$ if and only if the $(N+1)$ coefficients $Y(u_j)$ are contained in $[a,b]$.\\
\end{itemize}
\end{proposition}

\noindent The advantage of this model is that the infinite number of inequality constraints of $Y^N$ are \textit{equivalent} to a finite number of constraints on the coefficients $(Y(u_j))_{0\leq j\leq N}$. Therefore the problem is reduced to simulate the Gaussian vector $\boldsymbol{\xi}=(Y(u_0),\ldots,Y(u_N))^\top$ restricted to the convex subset formed by the two constraints \eqref{Gausconst1} and \eqref{Gausconst2}, where $C_{\boldsymbol{\xi}}=\{\boldsymbol{\xi}\in \mathbb{R}^{N+1}~: \ a\leq \xi_{j}\leq b, \ j=0,\ldots,N\}$. \\
 
\begin{proof}[Proof of Proposition~\ref{boundaryproposition}]
Since $Y(u_j), \ j=0,\ldots,N$ are Gaussian variables, then $Y^N$ is a GP with dimension
equal to $N+1$ and covariance function 
\begin{equation*}\label{finitecov}
\cov\left(Y^N(x),Y^N(x')\right)=\sum_{i,j=0}^N\cov\left(Y(u_i),Y(u_j)\right)h_i(x)h_j(x')=\sum_{i,j=0}^NK(u_i,u_j)h_i(x)h_j(x').
\end{equation*} 
To prove the pathwise convergence of $Y^N$ to $Y$, write more explicitly, for any $\omega\in\Omega$
\begin{equation*}
Y^N(x;\omega):=\sum_{j=0}^NY(u_j;\omega)h_j(x).
\end{equation*}
Hence, the sample paths of the approximating process $Y^N$ are piecewise linear approximations of the sample paths of the original process $Y$. From $h_j\geq 0$ and $\sum_{j=0}^Nh_j(x)=1$, for all $x\in [0,1]$, we get
\begin{eqnarray}
\left|Y^N(x;\omega)-Y(x;\omega)\right|&=&\left|\sum_{j=0}^N(Y(u_j;\omega)-Y(x;\omega)h_j(x)\right| \nonumber \\
&\leq & \sum_{j=0}^N \sup_{|x-x'|\leq \Delta_N}\left|Y(x';\omega)-Y(x;\omega)\right|h_j(x)=\sup_{|x-x'|\leq \Delta_N}\left|Y(x';\omega)-Y(x;\omega)\right|. \label{boundinequality}
\end{eqnarray}
By uniformly continuity of sample paths of the process $Y$ on the compact interval $[0,1]$, this last inequality \eqref{boundinequality} shows that
\begin{equation*}
\sup_{x\in [0,1]}\left|Y^N(x;\omega)-Y(x;\omega)\right|\underset{N\to +\infty}\longrightarrow 0
\end{equation*}
with probability 1. Now, if the $(N+1)$ coefficients $Y(u_j)_{0\leq j\leq N}$ are in the interval $[a,b]$ then the piecewise linear approximation $Y^N$ is in $C$. Conversely, suppose that $Y^N$ is in $C$ then 
\begin{equation*}\label{unifcontinuous}
Y^N(u_i)=\sum\limits_{j=0}^NY(u_j)h_j(u_i)=\sum\limits_{j=0}^NY(u_j)\delta_{ij}=Y(u_i)\in [a,b],
\end{equation*}
$i=0,\ldots,N$, which completes the proof of the last property, and hence
concludes the proof of the proposition. \qed
\end{proof}

\paragraph{Simulated paths.} As shown in Proposition~\ref{boundaryproposition}, the simulation of the finite-dimensional approximation of Gaussian processes $Y^N$ conditionally to given data and boundedness constraints ($Y^N\in I\cap C$) is reduced to simulate the Gaussian vector $\boldsymbol{\xi}$ restricted to $I_{\boldsymbol{\xi}}\cap C_{\boldsymbol{\xi}}$~:
\begin{eqnarray*}
&&I_{\boldsymbol{\xi}}=\left\{\boldsymbol{\xi}\in \mathbb{R}^{N+1}~: \ A\boldsymbol{\xi}=\boldsymbol{y}\right\},\\
&&C_{\boldsymbol{\xi}}=\left\{\boldsymbol{\xi}\in \mathbb{R}^{N+1}~: \ a\leq \xi_{j}\leq b, \ j=0,\ldots,N\right\},
\end{eqnarray*}
where the $n\times (N+1)$ matrix $A$ is defined as $A_{i,j}:=h_j\left(x^{(i)}\right)$. The interpolation system $A\boldsymbol{\xi}=\boldsymbol{y}$ admits solutions only if $N+1-n\geq 1$ (number of degrees of freedom).\\

The sampling scheme can be summarized in two steps~: first of all, we compute the conditional distribution of the Gaussian vector $\boldsymbol{\xi}$ with respect to data interpolation 
\begin{equation}\label{intXi}
\boldsymbol{\xi} \suchthat A\boldsymbol{\xi}=\boldsymbol{y} \sim \mathcal{N}\left(\left(A\Gamma^N\right)^\top\left(A\Gamma^NA^\top\right)^{-1}\boldsymbol{y},\Gamma^N-\left(A\Gamma^N\right)^\top\left(A\Gamma^NA^\top\right)^{-1}A\Gamma^N\right).
\end{equation}
Then, we simulate the Gaussian vector $\boldsymbol{\xi}$ with the above distribution \eqref{intXi} and, using an improved rejection sampling \cite{Maatouk2014}, we select only random coefficients in the convex set $[a,b]$. The sample paths of the conditional Gaussian process are generated by equation \eqref{boundaryapproach}, hence satisfy both interpolation conditions and boundedness constraints in the entire domain (see the R package developed in \cite{maatoukpackage2015} for more details).

\subsubsection{Monotonicity constraints}\label{MC} 
In this section, the real function $f$ is assumed to be of class $\mathcal{C}^1$. The convex set $C$ is the space of non-decreasing functions and is defined as 
\begin{equation*}
C:=\left\{f\in \mathcal{C}^1([0,1])~: \ f'(x)\geq 0, \ x\in[0,1]\right\}.
\end{equation*}
Since the monotonicity is related to the sign of the derivative, then the proposed model is adapted from model \eqref{boundaryapproach}. The basis functions are defined as the primitive functions of $h_j$, 
\begin{equation*}\label{monbasis}
\phi_j(x):=\int_0^xh_j(t)dt, \qquad x\in [0,1].
\end{equation*}
Remark that the derivative of the basis functions $\phi_j$ at any knot $u_i, \ i=0,\ldots,N$ is equal to the Kronecker's delta $(\phi'_j(u_i)=\delta_{ij})$. In Figure~\ref{phii}, we illustrate the basis functions $\phi_j$, $0\leq j\leq 4$. Notice that all these functions are non-decreasing and starting from $0$. In Figure~\ref{h2phi2}, we plot the basis function $\phi_2$ and the associate function $h_2$ for $N=4$. \\

\begin{figure}[hptb]
\begin{minipage}{.5\linewidth}
\centering
\subfloat[]{\label{phii}\includegraphics[scale=.3]{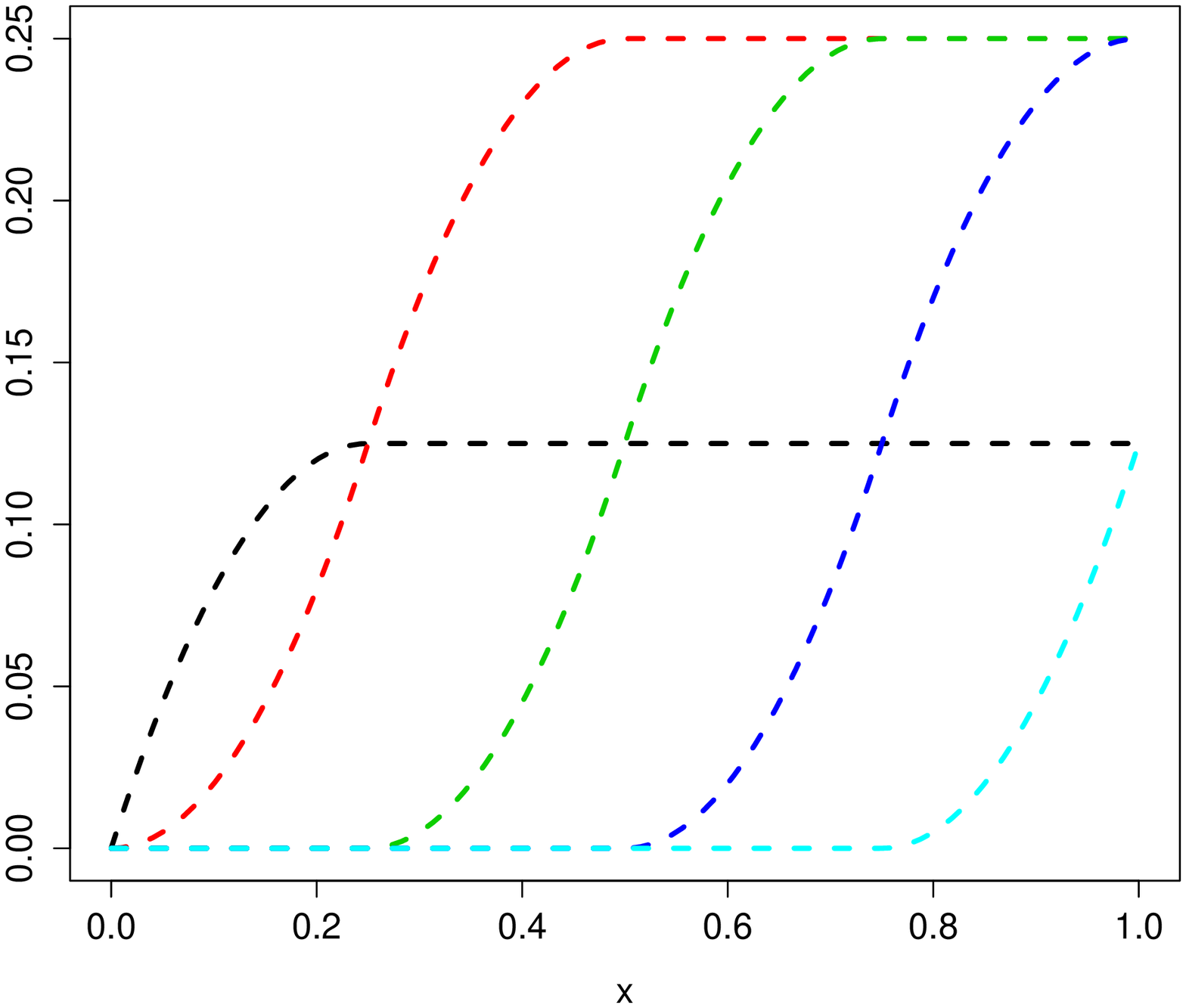}}
\end{minipage}%
\begin{minipage}{.5\linewidth}
\centering
\subfloat[]{\label{h2phi2}\includegraphics[scale=.3]{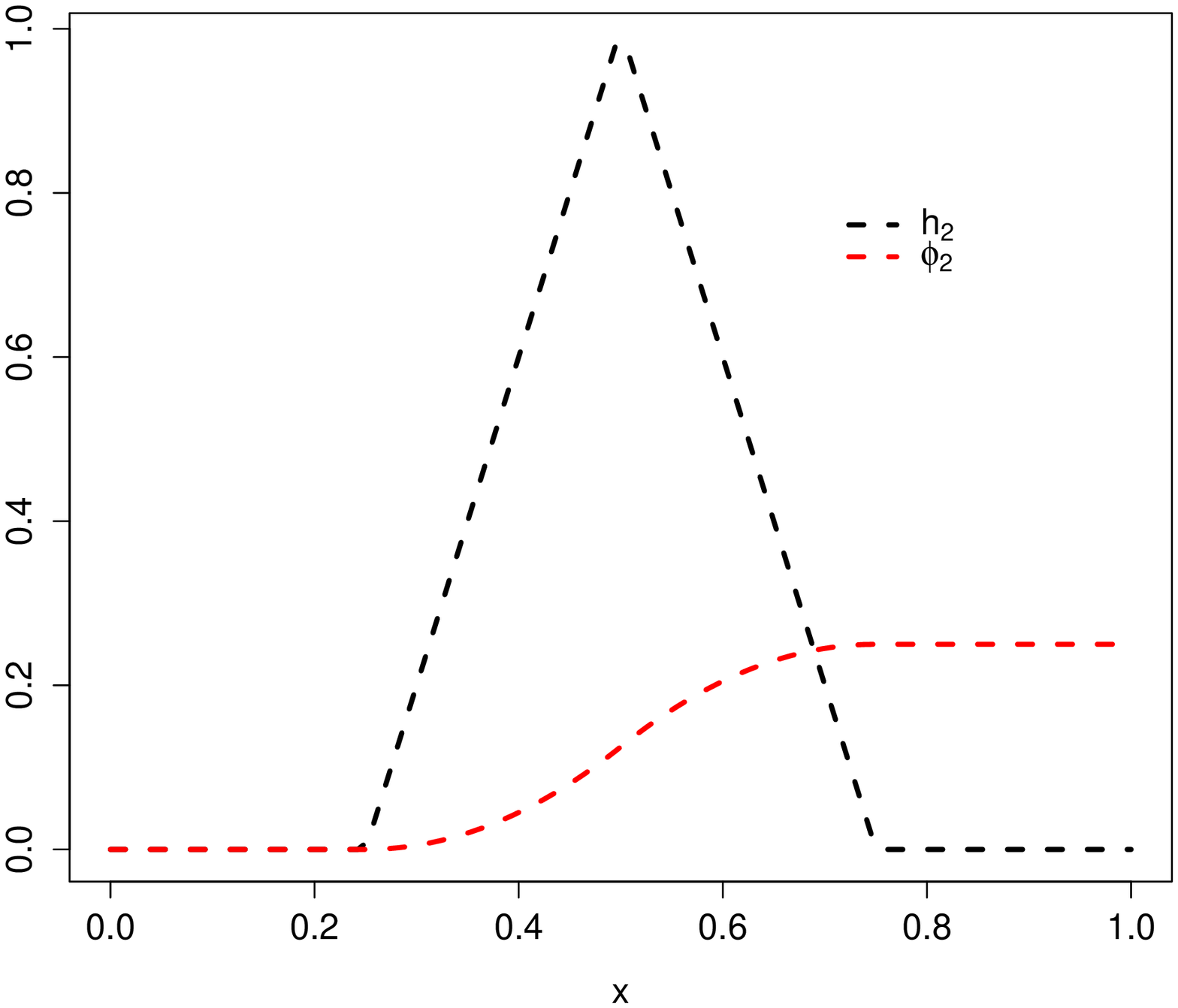}}
\end{minipage}
\caption{The basis functions $\phi_j, \ 0\leq j\leq 4$ (Figure~\ref{phii}) and the function $h_2$ with the corresponding function $\phi_2$ 
(Figure~\ref{h2phi2}).}
\label{hiphii}
\end{figure}

\noindent Similarly to Proposition~\ref{boundaryproposition}, we have the following results. \\
\begin{proposition}\label{monotonicityproposition}
Suppose that the realizations of the original GP $Y$ are almost surely continuously differentiable. Using the notations introduced before, the finite-dimensional approximation of Gaussian processes $(Y^N(x))_{x\in [0,1]}$ is defined as
\begin{equation}\label{monotonicityapproach}
Y^N(x):=Y(0)+\sum_{j=0}^NY'(u_j)\phi_j(x)=\zeta+\sum_{j=0}^N\xi_j\phi_j(x),
\end{equation}
where $\zeta=Y(0)$ and $\xi_j=Y'(u_j), \ j=0,\cdots,N$. Then we have the following properties~: \\ 

\begin{itemize}
\item $Y^N$ is a finite-dimensional GP with covariance function
{\em \begin{equation*}
K_N(x,x')=\left(1,\phi(x)^\top\right)\Gamma_{\text{new}}^N\left(1,\phi(x')^\top\right)^\top, 
\end{equation*}}
where $\phi(x)=(\phi_0(x),\ldots,\phi_N(x))^\top$ and {\em $\Gamma^N_{\text{new}}$} is the covariance matrix of the Gaussian vector $(\zeta,\boldsymbol{\xi})=\left(Y(0),Y'(u_0),\ldots,Y'(u_N)\right)^\top$ which is equal to~: 
{\em \begin{equation*}
\Gamma^N_{\text{new}}=\left[\begin{matrix}
K(0,0) & \frac{\partial K}{\partial x'}(0,u_j) \\\\
\frac{\partial K}{\partial x}(u_i,0) & \Gamma^N_{i,j}\\
\end{matrix}
\right]_{0\leq i,j\leq N},
\end{equation*}}
with $\Gamma^N_{i,j}=\frac{\partial^2K}{\partial x\partial x'}(u_i,u_j), \ i,j=0,\ldots,N$ and $K$ the covariance function of the original GP $Y$. \\

\item $Y^N$ converges uniformly to $Y$ when $N$ tends to infinity (with probability 1). \\

\item $Y^N$ is non-decreasing if and only if the coefficients $(Y'(u_j))_{0\leq j\leq N}$ are all nonnegative.\\
\end{itemize}
\end{proposition}

\noindent From the last property, the problem is reduced to simulate the Gaussian vector $(\zeta,\boldsymbol{\xi})$ restricted to the convex set formed by the interpolation conditions and the inequality constraints respectively, 
\begin{eqnarray*}
&&Y^N\left(x^{(i)}\right)=\zeta+\sum_{j=0}^N\xi_j\phi_j\left(x^{(i)}\right)=y_i, \qquad i=1,\ldots,n, \\
&& (\zeta,\boldsymbol{\xi})\in C_{\boldsymbol{\xi}}=\left\{\left(\zeta,\boldsymbol{\xi}\right)\in \mathbb{R}^{N+2}~: \ \xi_j\geq 0, \ j=0,\ldots,N\right\}. 
\end{eqnarray*}
\begin{proof}[Proof of Proposition~\ref{monotonicityproposition}]
The first property is a consequence of the fact that the derivative of a GP is also a GP. For all $x,x'\in [0,1]$, 
\begin{eqnarray*}\label{KNmon}
K_N(x,x')&=&\cov\left(Y^N(x),Y^N(x')\right)=\var\left(Y(0)\right)+\sum_{i=0}^N\frac{\partial K}{\partial x}(u_i,0)\phi_i(x)\\
&+&\sum_{j=0}^N\frac{\partial K}{\partial x'}(0,u_j)\phi_j(x)+\sum_{i,j=0}^N\frac{\partial^2 K}{\partial x\partial x'}(u_i,u_j)\phi_i(x)\phi_j(x').
\end{eqnarray*}
To prove the pathwise convergence of $Y^N$ to $Y$, let us write that for any $\omega\in\Omega$,
\begin{equation*}
Y^N(x;\omega)=Y(0;\omega)+\int_0^x\left(\sum_{j=0}^NY'(u_j;\omega)h_j(t)\right)dt.
\end{equation*}
From Proposition~\ref{boundaryproposition}, $\sum_{j=0}^NY'(u_j;\omega)h_j(x)$ converges uniformly pathwise to $Y'(x;\omega)$ since the realizations of the process are almost surely continuously differentiable. One can conclude that $Y^N$ converges uniformly to $Y$ for almost all $\omega\in \Omega$. Now, if $Y'(u_j), \ j=0,\ldots,N$ are all nonnegative then $Y^N$ is non-decreasing
since the basis functions $(\phi_j)_{0\leq j\leq N}$ are non-decreasing. Conversely, if $Y^N$ is non-decreasing, we have
\begin{equation*}
0\leq \left(Y^N\right)'(u_i)=\sum_{j=0}^NY'(u_j)h_j(u_i)=Y'(u_j),
\end{equation*}
$i=0,\ldots,N$, which completes the proof of the last property and the proposition.  \qed 
\end{proof} 

\paragraph{Simulated paths.} As shown in Proposition~\ref{monotonicityproposition}, the simulation of the finite-dimensional approximation of Gaussian processes $Y^N$ conditionally to given data and monotonicity constraints ($Y^N\in I\cap C$) is reduced to simulate the Gaussian vector $(\zeta,\boldsymbol{\xi})$ restricted to $I_{\boldsymbol{\xi}}\cap C_{\boldsymbol{\xi}}$~:
\begin{eqnarray*}
&&I_{\boldsymbol{\xi}}=\left\{(\zeta,\boldsymbol{\xi})\in\mathbb{R}^{N+2}~: \ A(\zeta,\boldsymbol{\xi})=\boldsymbol{y}\right\},\\
&&C_{\boldsymbol{\xi}}=\left\{(\zeta,\boldsymbol{\xi})\in\mathbb{R}^{N+2}~: \ \xi_j\geq 0, \ j=0,\ldots,N\right\},
\end{eqnarray*}
where the $n\times (N+2)$ matrix $A$ is defined as 
\begin{eqnarray*}
A_{i,j}:=\left\{\begin{array}{ll}
1 & \mbox{for} \ i=1,\ldots,n \ \text{and} \ j=1,\\
\phi_{j-2}\left(x^{(i)}\right) & \mbox{for} \ i=1,\ldots,n \ \text{and} \ j=2,\ldots,N+2.
\end{array}\right.
\end{eqnarray*}
We simulate the Gaussian vector $(\zeta,\boldsymbol{\xi})$ with the conditional distribution defined in \eqref{intXi}, where $\Gamma^N$ is replaced by $\Gamma^N_{\text{new}}$. Then, using an improved rejection sampling \cite{Maatouk2014}, we select the nonnegative coefficients $\xi_j$. Finally, the sample paths of the conditional Gaussian process are generated by equation \eqref{monotonicityapproach} which satisfy both interpolation conditions and monotonicity constraints in the entire domain.
 \\

\begin{remark}[Monotonicity of continuous but non-derivable functions]\label{mon1DC0}
{\rm If the real function is of class $\mathcal{C}^0$ only (but possibly not derivable) and non-decreasing in the entire domain, then the proposed model defined in \eqref{boundaryapproach} is non-decreasing \textit{if and only if} the sequence of coefficients $(Y(u_j))_j,\ j=0,\ldots,N$ is non-decreasing (i.e. $Y(u_{j-1})\leq Y(u_j), \ j=1,\ldots,N$). The simulated paths are generated using the same strategy in Sect.~\ref{BPCOD}, where $C_{\boldsymbol{\xi}}=\{\boldsymbol{\xi}\in \mathbb{R}^{N+1}~: \ \xi_{j-1}\leq \xi_{j}, \ j=1,\ldots,N\}$.}
\end{remark}

\subsubsection{Convexity constraints}\label{CC}
In this section, the real function is supposed to be two times differentiable. Since the functions $h_j, \ j=0,\ldots,N$ defined in \eqref{hj} are all nonnegative, then the basis functions $\varphi_j$ are taken as the two times primitive functions of $h_j$, 
\begin{equation*}
\varphi_j(x) :=\int_0^x\left(\int_0^th_j(u)du\right)dt.
\end{equation*}

In Figure~\ref{convexB}, we illustrate the basis functions $\varphi_j, \ (0\leq j\leq 4)$. Notice that all these functions are convex and pass through the origin. Moreover, the derivatives at the origin are equal to zero. In Figure~\ref{phi2h2}, we illustrate the basis function $\varphi_2$ and the associate function $h_2$.

\begin{figure}[hptb]
\begin{minipage}{.5\linewidth}
\centering
\subfloat[]{\label{convexB}\includegraphics[scale=.3]{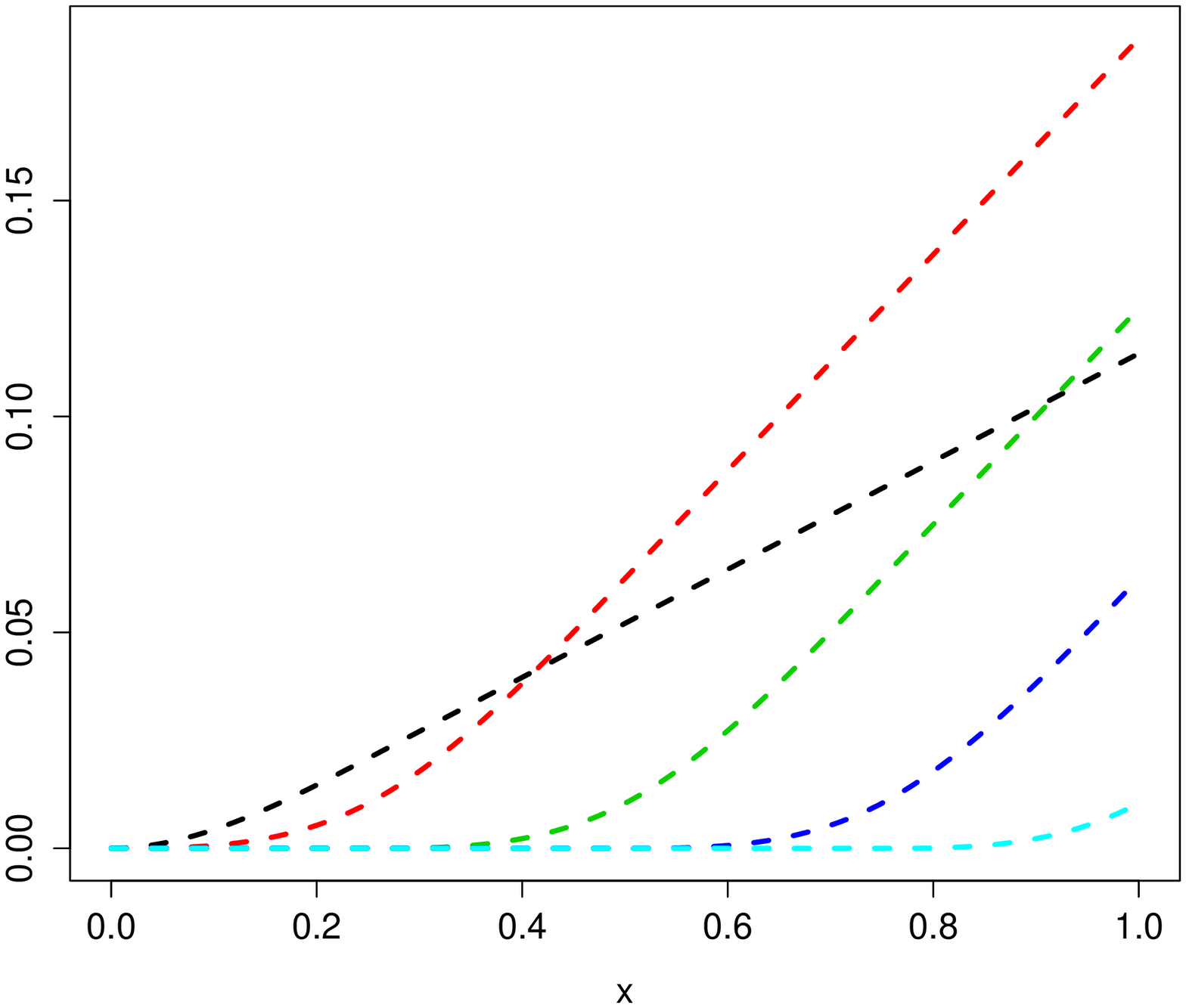}}
\end{minipage}%
\begin{minipage}{.5\linewidth}
\centering
\subfloat[]{\label{phi2h2}\includegraphics[scale=.31]{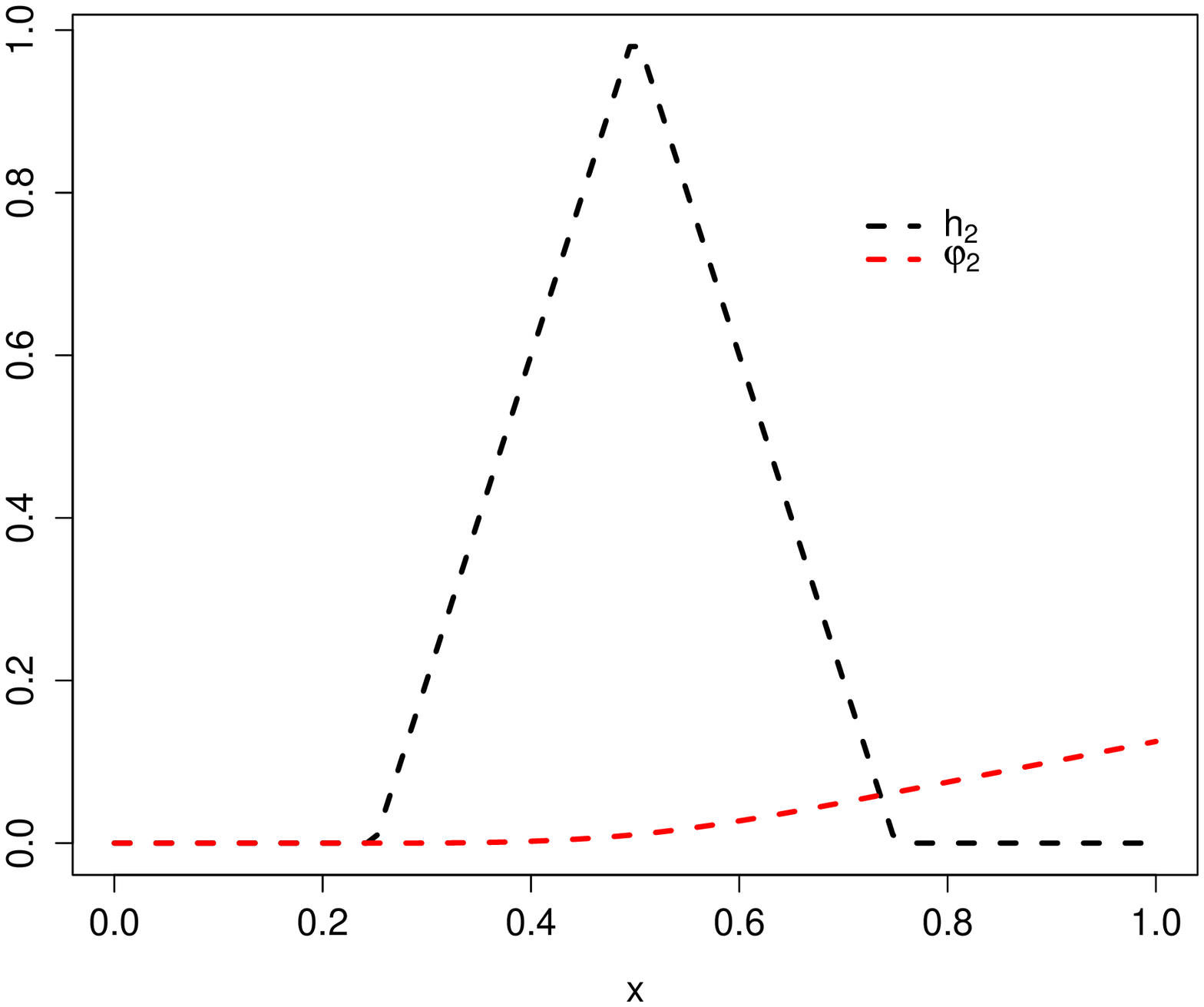}}
\end{minipage}
\caption{The basis functions $\varphi_j, \ 0\leq j \leq 4$ (Figure~\ref{convexB}) and the function $h_2$ with the corresponding function $\varphi_2$ (Figure~\ref{phi2h2}).}
\label{phiihi}
\end{figure}

Similarly to the monotonicity case, the second derivative of the basis functions $\varphi''_j$ at any knot $u_i, \ (0\leq i\leq N)$ is equal to Kronecker's delta $(\varphi''_j(u_i)=\delta_{ij})$. We assume here that the realizations of the original GP $Y$ are at least two times differentiable. The finite-dimensional approximation defined as 
\begin{equation}\label{convexityapproach}
Y^N(x):=Y(0)+Y'(0)x+\sum_{j=0}^NY''(u_j)\varphi_j(x)=\zeta+\kappa x+\sum_{j=0}^N\xi_j\varphi_j(x), 
\end{equation}
is convex \textit{if and only if} the $(N+1)$ random coefficients $\xi_j=Y''(u_j)$ are all nonnegative,
where $\zeta=Y(0)$ and $\kappa=Y'(0)$. Thus, the problem is reduced to generate the Gaussian vector 
$(\zeta,\kappa,\boldsymbol{\xi})=\left(Y(0),Y'(0),Y''(u_0),\ldots,Y''(u_N)\right)^\top$ restricted to the convex set $I_{\boldsymbol{\xi}}\cap C_{\boldsymbol{\xi}}$, where
\begin{eqnarray*}
&&I_{\boldsymbol{\xi}}=\left\{(\zeta,\kappa,\boldsymbol{\xi})\in\mathbb{R}^{N+3}~: \ \zeta+\kappa x^{(i)}+\sum_{j=0}^N\xi_j\varphi_j\left(x^{(i)}\right)=y_i, \ i=1,\ldots,n\right\}, \\
&&C_{\boldsymbol{\xi}}=\left\{(\zeta,\kappa,\boldsymbol{\xi})\in \mathbb{R}^{N+3}~: \ \xi_j \geq 0, \ j=0,\ldots,N\right\}.
\end{eqnarray*}
Its $(N+3)\times(N+3)$ covariance matrix $\Gamma^N_{\text{new}}$ is given by
\begin{equation*}
\Gamma^N_{\text{new}}=\left[\begin{matrix}
\var(\zeta) & \cov(\zeta,\kappa) & \cov(\zeta,\boldsymbol{\xi}) \\\\
\cov(\kappa,\zeta) & \var(\kappa) & \cov(\kappa,\boldsymbol{\xi}) \\\\
\cov(\boldsymbol{\xi},\zeta) &  \cov(\boldsymbol{\xi},\kappa) & \cov(\boldsymbol{\xi},\boldsymbol{\xi}) \\
\end{matrix}
\right]=\left[\begin{matrix}
K(0,0) & \frac{\partial K}{\partial x'}(0,0) & \frac{\partial^2 K}{\partial (x')^2}(0,u_j)\\\\
\frac{\partial K}{\partial x}(0,0) & \frac{\partial^2K}{\partial x\partial x'}(0,0) & 
\frac{\partial^3K}{\partial x\partial (x')^2}(0,u_j)\\\\
\frac{\partial^2 K}{\partial x^2}(u_i,0) & \frac{\partial^3K}{\partial x^2\partial x'}(u_i,0) &\Gamma^N_{i,j}
\end{matrix}
\right]_{0\leq i,j\leq N},
\end{equation*}
\noindent where
\begin{equation*}
\Gamma^N_{i,j}=\cov(\xi_i,\xi_j)=\cov(Y''(u_i),Y''(u_j))=\frac{\partial^4K}{\partial x^2\partial (x')^2}(u_i,u_j), \qquad i,j=0,\ldots,N.
\end{equation*}
Finally, the covariance function of the finite-dimensional approximation of GPs is equal to~:
\begin{equation*}
K_N(x,x')=\left(1,x,\varphi(x)^\top\right)\Gamma_{\text{new}}^N\left(1,x',\varphi(x')^\top\right)^\top, 
\end{equation*}
where $\varphi(x)=\left(\varphi_0(x),\ldots,\varphi_N(x)\right)^\top$. \\

\paragraph{Simulated paths.} As shown in this section, the simulation of the finite-dimensional approximation of Gaussian processes $Y^N$ conditionally to given data and convexity constraints ($Y^N\in I\cap C$) is reduced to simulate the Gaussian vector $(\zeta,\kappa,\boldsymbol{\xi})$ restricted to $I_{\boldsymbol{\xi}}\cap C_{\boldsymbol{\xi}}$~:
\begin{eqnarray*}
&&I_{\boldsymbol{\xi}}=\left\{(\zeta,\kappa,\boldsymbol{\xi})\in\mathbb{R}^{N+3}~: \ A(\zeta,\kappa,\boldsymbol{\xi})=\boldsymbol{y}\right\},\\
&&C_{\boldsymbol{\xi}}=\left\{(\zeta,\kappa,\boldsymbol{\xi})\in\mathbb{R}^{N+3}~: \ \xi_j\geq 0, \ j=0,\ldots,N\right\},
\end{eqnarray*}
where the $n\times (N+3)$ matrix $A$ is defined as 
\begin{eqnarray*}
A_{i,j}:=\left\{\begin{array}{ll}
1 & \mbox{for} \ i=1,\ldots,n \ \text{and} \ j=1,\\
x^{(i)} & \mbox{for} \ i=1,\ldots,n \ \text{and} \ j=2,\\
\varphi_{j-3}\left(x^{(i)}\right) & \mbox{for} \ i=1,\ldots,n \ \text{and} \ j=3,\ldots,N+3.
\end{array}\right.
\end{eqnarray*}
We simulate the Gaussian vector $(\zeta,\kappa,\boldsymbol{\xi})$ with the conditional distribution defined in \eqref{intXi}, where $\Gamma^N$ is replaced by $\Gamma^N_{\text{new}}$. Then, using an improved rejection sampling \cite{Maatouk2014}, we select the nonnegative coefficients $\xi_j$. Finally, the sample paths of the conditional Gaussian process are generated by equation \eqref{convexityapproach} which satisfy both interpolation conditions and convexity constraints in the entire domain.
 \\

Now, we consider the problem dimension $d\geq 2$. For boundedness constraints,
our model can be easily extended to multidimensional cases. In the following, we are interested in studying isotonicity constraints. 

\subsection{Isotonicity in two dimensions}\label{MTD}
We now assume that the input is $\boldsymbol{x}=(x_1,x_2)\in\mathbb{R}^2$ and without loss of generality is in the unit square. The real function $f$ is supposed to be monotone (non-decreasing for example) with respect to the two input variables~:
\begin{equation*}
x_1\leq x'_1 \mbox{\quad and \quad} x_2\leq x'_2 \quad \Rightarrow \quad f(x_1,x_2)\leq f(x'_1,x'_2).
\end{equation*}
The idea is the same as the one-dimensional case. We construct the basis functions such that monotonicity constraints
are \textit{equivalent} to constraints on the coefficients. Firstly, 
we discretize the unit square (e.g. uniformly to $(N+1)^2$ knots, see below Figure~\ref{grid} for $N=7$). Secondly, on each knot we build a basis function. For instance, the basis function at the knot $(u_i,u_j)$ is defined as 
\begin{equation*}
\Phi_{i,j}(\boldsymbol{x}):=h_i(x_1)h_j(x_2),
\end{equation*}
where $h_j, \ j=0,\ldots,N$ are defined in \eqref{hj}. We have
\begin{equation*}
\Phi_{i,j}(u_k,u_{\ell})=\delta_{i,k}\delta_{j,\ell}, \qquad k,\ell=0,\ldots,N. 
\end{equation*}
\\

\begin{proposition}\label{mono2Dproposition}
Using the notations introduced before, the finite-dimensional approximation of Gaussian processes $(Y^N(\boldsymbol{x}))_{\boldsymbol{x}\in [0,1]^2}$ is defined as
\begin{eqnarray}\label{monotonicity2Dapproach}
Y^N(x_1,x_2)&:=&\sum_{i,j=0}^NY(u_i,u_j)h_i(x_1)h_j(x_2)=\sum_{i,j=0}^N\xi_{i,j}h_i(x_1)h_j(x_2),
\end{eqnarray}
where $\xi_{i,j}=Y(u_i,u_j)$ and the functions $h_j, \ j=0,\ldots,N$ are defined in \eqref{hj}. Then, we have the following properties~: \\

\begin{itemize}
\item $Y^N$ is a finite-dimensional GP with covariance function
$K_N(\boldsymbol{x},\boldsymbol{x}')=\Phi(\boldsymbol{x})^\top\Gamma^N \Phi(\boldsymbol{x}')$, where $\Phi(\boldsymbol{x})^\top=\left(h_i(x_1)h_j(x_2)\right)_{i,j}$, $\Gamma^N_{(i,j),(i',j')}=K\left((u_i,u_j),(u_{i'},u_{j'})\right)$ and $K$ is the covariance function of the original GP $Y$. \\

\item $Y^N$ converges uniformly to $Y$ when $N$ tends to infinity (with probability 1). \\

\item $Y^N$ is non-decreasing with respect to the two input variables \textit{if and only if} the $(N+1)^2$ random coefficients $\xi_{i,j}, \ i,j=0,\ldots,N$ verify the following linear constraints~: \\
{\em 
\begin{enumerate}
\item $\xi_{i-1,j}\leq \xi_{i,j} 
\mbox{ and } \xi_{i,j-1}\leq \xi_{i,j}, \ i,j=1,\ldots,N$.
\item $\xi_{i-1,0}\leq \xi_{i,0}, \ i=1,\ldots,N$.
\item $\xi_{0,j-1}\leq \xi_{0,j}, \ j=1,\ldots,N$.\\
\end{enumerate} }
\end{itemize} 
\end{proposition}

\noindent From the last property, the problem is reduced to simulate the Gaussian vector $\boldsymbol{\xi}=(\xi_{i,j})_{i,j}$ restricted to the convex set $I_{\boldsymbol{\xi}}\cap C_{\boldsymbol{\xi}}$, where 
\begin{eqnarray*}
&&I_{\boldsymbol{\xi}}=\left\{\boldsymbol{\xi}\in\mathbb{R}^{(N+1)^2}~: \ Y^N\left(x_1^{(i)},x_2^{(i)}\right)=\sum_{i,j=0}^N\xi_{i,j}h_i\left(x_1^{(i)}\right)h_j\left(x_2^{(i)}\right)=y_i\right\}, \\
&&C_{\boldsymbol{\xi}}=\left\{\boldsymbol{\xi}\in\mathbb{R}^{(N+1)^2} \mbox{such that $\xi_{i,j}$ verify the constraints 1. 2. and 3.}\right\}. \\
\end{eqnarray*}

\begin{proof}[Proof of Proposition~\ref{monotonicity2Dapproach}]
The proof of the first two properties is similar to the one-dimensional case. Now,
if the $(N+1)^2$ coefficients $\xi_{i,j},\ i,j=0,\cdots,N$ verify the above linear constraints 1. 2. and 3. then $Y^N$ is non-decreasing since $Y^N$ is a piecewise linear function for $x_1$ or $x_2$ directions. Conversely, if $Y^N$ is non-decreasing then $Y^N(u_i,u_j)=\xi_{i,j}, \ i,j=0,\ldots,N$ satisfy the constraints 1. 2. and 3. . \qed
\end{proof} 

\begin{remark}[Isotonicity in two dimensions with respect to one variable]\label{rem2D}
{\rm If the function is non-decreasing with respect to the first variable only, then the proposed GP defined as
\begin{equation}\label{monotonicity2Dx1}
Y^N(x_1,x_2):=\sum_{i,j=0}^NY(u_i,u_j)h_i(x_1)h_j(x_2)=
\sum_{i,j=0}^N\xi_{i,j}h_i(x_1)h_j(x_2),
\end{equation}
is non-decreasing with respect to $x_1$ \textit{if and only if} the random coefficients $\xi_{i-1,j}\leq \xi_{i,j}, \ i=1,\ldots,N \ \mbox{and} \ j=0,\ldots,N$. }
\end{remark}

\subsection{Isotonicity in multidimensional cases}\label{MMC}
The d-dimensional case is a simple extension of the two-dimensional case.
The finite-dimensional approximation of Gaussian processes $Y^N$ can be written as 
\begin{equation*}
Y^N(\boldsymbol{x}):=\sum_{i_1,\ldots,i_d=0}^N Y(u_{i_1},\ldots,u_{i_d})\prod_{\sigma\in\{1,\ldots,d\}}h_{i_{\sigma}}(x_{\sigma})
=\sum_{i_1,\ldots,i_d=0}^N \xi_{i_1,\ldots,i_d}\prod_{\sigma\in\{1,\ldots,d\}}h_{i_{\sigma}}(x_{\sigma}),
\end{equation*} 
where $\xi_{i_1,\ldots,i_d}=Y(u_{i_1},\ldots,u_{i_d})$. Remark~\ref{rem2D} can be extended as well, for the case of a monotonicity with respect to a subset of variables. For instance, the monotonicity of $Y^N$ with respect to the $\ell^{\text{th}}$ dimension input $x_{\ell}$ is \textit{equivalent} to the fact that $\xi_{i_1,\ldots,i_{\ell}-1,\ldots,i_d}\leq \xi_{i_1,\ldots,i_{\ell},\ldots,i_d}, \ i_{\ell}=1,\ldots,N$ and $i_1,\ldots,i_{\ell-1},i_{\ell+1},\ldots,i_d=0,\ldots,N$. \\
 
\subsection{Simulation of GPs conditionally to equality and inequality constraints}
For the sake of simplicity and without loss of generality, we suppose that the proposed finite-dimensional approximation of GPs is of the form 
\begin{equation*}
Y^N(\boldsymbol{x})=\sum_{j=0}^N\xi_j\phi_j(\boldsymbol{x}), \qquad \boldsymbol{x}\in \mathbb{R}^d,
\end{equation*} 
where $\boldsymbol{\xi}=(\xi_0,\ldots,\xi_N)^\top$ is a zero-mean Gaussian vector with covariance matrix $\Gamma^N$ and $\phi=(\phi_0,\ldots,\phi_N)^\top$ are deterministic basis functions. For instance, the constant term $Y(0)$ in model \eqref{monotonicityapproach} can be written as $\xi_0\phi_0(x)$, where $\phi_0(x)=1$.
The space of interpolation conditions is $I_{\boldsymbol{\xi}}=\left\{\boldsymbol{\xi}\in\mathbb{R}^{N+1}~:  \sum_{j=0}^N\xi_j\phi_j\left(\boldsymbol{x}^{(i)}\right)=y_i, \ i=1,\ldots,n\right\}$ and the set of inequality constraints $C_{\boldsymbol{\xi}}$ is a convex set (for instance, the nonnegative quadrant $\xi_j\geq 0, \ j=0,\ldots,N$ for non-decreasing constraints in one dimension). We are interested in the calculation of the mean, mode (maximum \textit{a posteriori}) of $Y^N$ conditionally to $\boldsymbol{\xi}\in I_{\boldsymbol{\xi}}\cap C_{\boldsymbol{\xi}}$ and in the quantification of prediction intervals. Note that their analytical forms except for the mode are not easy to find, hence we need simulation. As explained in Sect.~\ref{BPCOD}, the problem is reduced to simulate the Gaussian vector $\boldsymbol{\xi}=(\xi_0,\ldots,\xi_N)^\top$ restricted to convex sets. In that case, several algorithms can be used (see e.g. \cite{Botts}, \cite{Chopin2011FST19607241960748}, \cite{Geweke91efficientsimulation}, \cite{756335}, \cite{Maatouk2014}, \cite{journals/sac/PhilippeR03} and \cite{Robert}). \\

In this section, we introduce some notations that will be used in Sect.~\ref{illustrative}, and emphasize the two cases of truncated simulations. We note $\boldsymbol{\xi}_{\text{I}}$ the mean of $\boldsymbol{\xi}$ conditionally to $\boldsymbol{\xi}\in I_{\boldsymbol{\xi}}$ without inequality constraints (see equation \eqref{intXi}). Then by linearity of the conditional expectation, the so-called usual (unconstrained) Kriging mean is equal to 
\begin{eqnarray*}
m^N_{\text{K}}(\boldsymbol{x}):=\mathds{E}\left(Y^N(\boldsymbol{x}) \suchthat Y^N\left(\boldsymbol{x}^{(i)}\right)=y_i, \ i=1,\ldots,n\right)=\sum_{j=0}^N(\boldsymbol{\xi}_{\text{I}})_j\phi_j(\boldsymbol{x}),
\end{eqnarray*} 
where $\boldsymbol{\xi}_{\text{I}}=\mathds{E}\left(\boldsymbol{\xi} \suchthat \boldsymbol{\xi}\in I_{\xi}\right)=\Gamma^NA^\top\left(A\Gamma^NA^\top\right)^{-1}\boldsymbol{y}\in\mathbb{R}^{N+1}$ and the matrix $A$ is formed by the values of the basis functions at the observations (i.e. $A_{i,j}=\phi_j\left(x^{(i)}\right)$). Similarly to the Kriging mean of the original GP (see equation \eqref{condGP}), the Kriging mean $m_{\text{K}}^N$ of the finite-dimensional approximation of GPs $Y^N$ can be written as
\begin{equation*}
m_{\text{K}}^N(\boldsymbol{x})=\boldsymbol{k}_N(\boldsymbol{x})^\top\mathbb{K}_N^{-1}\boldsymbol{y},
\end{equation*}
where $\boldsymbol{k}_N(\boldsymbol{x})=\left(K_N\left(\boldsymbol{x},\boldsymbol{x}^{(i)}\right)\right)_i=\left(A\Gamma^N\phi(\boldsymbol{x})\right)$ is the vector of covariance between $Y^N(\boldsymbol{x})$ and $Y^N\left(\boldsymbol{X}\right)$ and $(\mathbb{K}_N)_{i,j}=K_N\left(\boldsymbol{x}^{(i)},\boldsymbol{x}^{(j)}\right)_{i,j}=\left(A\Gamma^NA^\top\right)$, $i,j=1,\ldots,n$ is the covariance matrix of $Y^N\left(\boldsymbol{X}\right)=\boldsymbol{y}$. \\
 
\begin{definition}\label{IKMdef}
Denote {\em $\boldsymbol{\xi}_{\text{C}}$} as the mean of the Gaussian vector $\boldsymbol{\xi}$ restricted to $I_{\boldsymbol{\xi}}\cap C_{\boldsymbol{\xi}}$ (i.e. the posterior mean). Then, the inequality Kriging mean (mean \textit{a posteriori}) is defined as
{\em \begin{equation*}
m^N_{\text{IK}}(\boldsymbol{x}):=\mathds{E}\left(Y^N(\boldsymbol{x})\suchthat Y^N\left(\boldsymbol{x}^{(i)}\right)=y_i,
\ \boldsymbol{\xi}\in C_{\boldsymbol{\xi}}\right)=\sum_{j=0}^N(\boldsymbol{\xi}_{\text{C}})_j\phi_j(\boldsymbol{x}),
\end{equation*}}
where {\em $\boldsymbol{\xi}_{\text{C}}=\mathds{E}\left(\boldsymbol{\xi} \suchthat \boldsymbol{\xi}\in I_{\boldsymbol{\xi}}\cap C_{\boldsymbol{\xi}}\right)$}.
\end{definition}

\noindent Finally, let $\mu$ be the maximum of the probability density function (pdf) of $\boldsymbol{\xi}$ restricted to $I_{\boldsymbol{\xi}}\cap C_{\boldsymbol{\xi}}$. It is the solution of the following convex optimization problem
\begin{equation}\label{mu}
\mu:=\arg\min_{\boldsymbol{c}\in I_{\boldsymbol{\xi}}\cap C_{\boldsymbol{\xi}}}\left(\frac{1}{2}\boldsymbol{c}^\top\left(\Gamma^N\right)^{-1}\boldsymbol{c}\right),
\end{equation}
where $\Gamma^N$ is the covariance matrix of the Gaussian vector $\boldsymbol{\xi}$.
In fact, $\mu$ corresponds to the mode\footnote{The maximum of the probability density function.} of the Gaussian vector $\boldsymbol{\xi}$ restricted to $I_{\boldsymbol{\xi}}\cap C_{\boldsymbol{\xi}}$ and its numerical calculation is a standard problem in the minimization of positive quadratic forms subject
to convex constraints, see e.g. \cite{Boyd2004CO993483} and \cite{Goldfarb83}. Let us mention that in all simulation examples illustrated in this paper, the R Package `solve.QP' described in \cite{goldfarb1982dual} and \cite{Goldfarb83} is used to compute the mode of the truncated Gaussian vector (i.e. to solve the quadratic convex optimization problem \eqref{mu}).\\ 

\begin{definition}\label{Mdef}
The so-called inequality mode or Maximum \textit{A Posteriori} (MAP) of the finite-dimensional approximation of GPs $Y^N$ conditionally to given data and inequality constraints is equal to
{\em \begin{equation*}
M^N_{\text{IK}}(\boldsymbol{x}):=\sum_{j=0}^N\mu_j\phi_j(\boldsymbol{x}),\qquad \boldsymbol{x}\in\mathbb{R}^d,
\end{equation*} }
where $\mu=(\mu_0,\ldots,\mu_N)^\top$ is defined in \eqref{mu}. \\
\end{definition}

\begin{remark}
The inequality mode {\em $M^N_{\text{IK}}$} defined in Definition~\ref{Mdef} does not depend on the variance parameter $\sigma$ of the covariance function $K$ since the vector $\mu$ and the basis functions $\phi_j$ do not depend on it as well. Also, it does not depend on the simulation but on the length hyper-parameters of the covariance function $\boldsymbol{\theta}=(\theta_1,\ldots,\theta_d)$.  \\
\end{remark}

\begin{remark}\label{correspondRemark}
The inequality mode or MAP (Maximum \textit{A Posteriori}) of the conditional GP {\em $M^N_{\text{IK}}$} converges uniformly to the \textit{constrained} interpolation function defined as the solution of the following convex optimization problem~:
\begin{equation*}
\arg\min_{h\in H\cap I\cap C}\|h\|_H^2,
\end{equation*}
where $H$ is a Reproducing Kernel Hilbert Space (RKHS) associated to the positive type kernel $K$ \cite{aronszajn1950}, $I$ is the set of functions verify interpolation conditions and the convex set $C$ is the space of functions which verify the inequality constraints (see \cite{bayhal-01136466}, \cite{2016arXiv160202714B} and \cite{Maatoukthesis2015} for more details). \

This extends to the case of interpolation conditions and inequality constraints the correspondence established by Kimeldorf and Wahba \cite{KW1970} between Bayesian estimation on stochastic process and smoothing by splines.  \\
\end{remark}

\begin{figure}[hptb]
\begin{minipage}{.5\linewidth}
\centering
\subfloat[]{\label{cas1}\includegraphics[scale=.31]{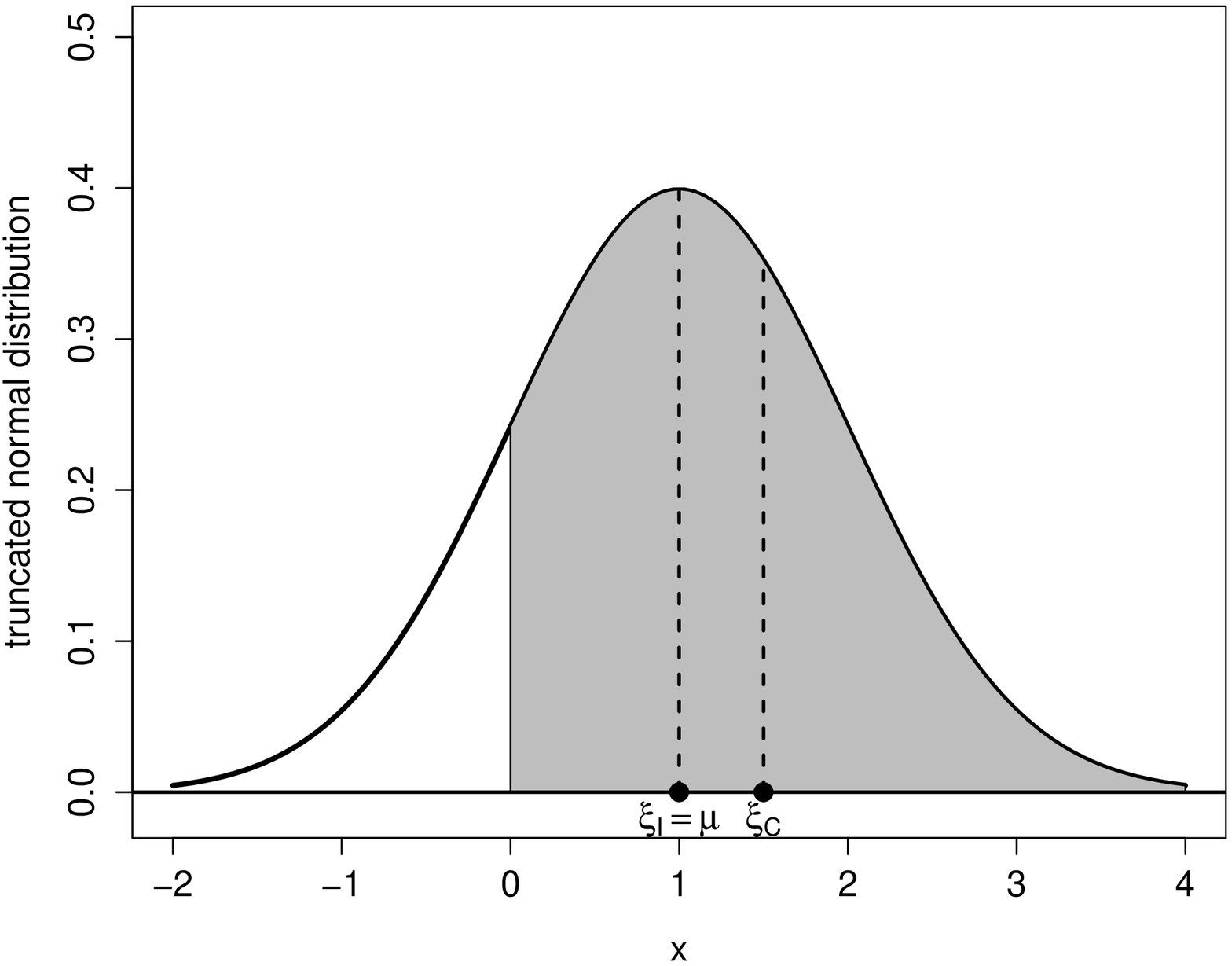}}
\end{minipage}%
\begin{minipage}{.5\linewidth}
\centering
\subfloat[]{\label{cas2}\includegraphics[scale=.31]{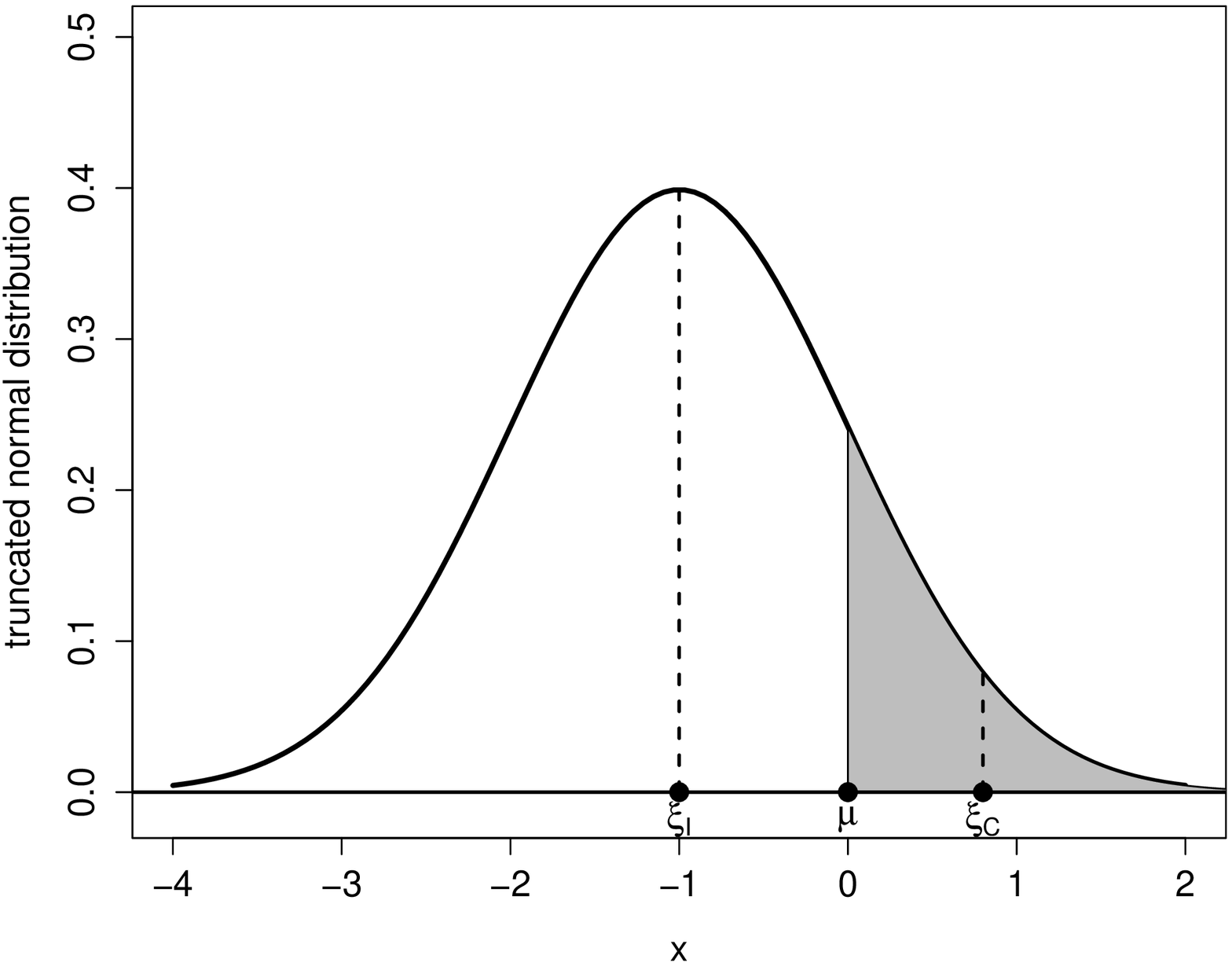}}
\end{minipage}
\caption{Two cases of truncated normal variables. The mean is inside (respectively outside) the acceptance region Figure~\ref{cas1} (respectively Figure~\ref{cas2}).} 
\label{fig:main}
\end{figure}

In practice, we have two cases in the simulation of truncated multivariate normal distributions (see Figures~\ref{cas1} and \ref{cas2} for example in one dimension). In the first case (Figure~\ref{cas1}), we have $\boldsymbol{\xi}_{\text{I}}=\mu$ and so $m^N_{\text{K}}=M^N_{\text{IK}}$ and they are different from $m^N_{\text{IK}}$. In this case, $\boldsymbol{\xi}_{\text{I}}$ is inside $C_{\boldsymbol{\xi}}$ (for instance the nonnegative quadrant) and the usual (unconstrained) Kriging mean respects the inequality constraints. The second one, where the three are different (Figure~\ref{cas2}). In this case, $\boldsymbol{\xi}_{\text{I}}$ is outside $C_{\boldsymbol{\xi}}$ and the usual (unconstrained) Kriging mean does not respect the inequality constraints. \\

\section{Simulation study}\label{SimStudy}
The aim of this section is to illustrate the performance of the proposed model in terms of prediction
and uncertainty quantification. To do this, we take the real increasing function $f(x)=\log(20x+1)$ used in \cite{golchi2015monotone} (black lines in Figure~\ref{GolchiExample}). Suppose that $f$ is evaluated at $X=(0,0.1,0.2,0.3,0.4,0.9,1)$. As mentioned in \cite{golchi2015monotone}, this is a challenge situation for unconstrained GP since we have a large gap between the fifth and sixth design points (i.e. $0.4<x<0.9$). In Figure~\ref{krigingGolchi}, the sample paths are taken from unconstrained GP using the Mat\'ern 5/2 covariance function (see Table~\ref{kernel}), where the hyper-parameters $\sigma$ and $\theta$ are estimated by the Maximum Likelihood Estimator (MLE) \cite{RoustantGinsbourgerDeville2012JSSOBKv51i01}. Notice that the simulated paths are not monotone and the prediction interval is quite large between 0.4 and 0.9 (Figure~\ref{krigingGolchi1}). In Figure~\ref{Golchi1D1}, prediction intervals and inequality mode taken from model \eqref{monotonicityapproach} conditionally to given data and monotonicity constraints are shown. The Mat\'ern 5/2 covariance function is used. Applying a suited cross validation method to estimate covariance hyper-parameters \cite{2016arXiv160402237C} and \cite{Maatouk201538}, we get $\sigma=335.5$ and $\theta=4.7$. The predictive uncertainty is reduced (Figure~\ref{Golchi1D1}). Furthermore, contrarily to the model described in \cite{golchi2015monotone}, we do not need to add derivative points to ensure monotonicity constraints in the entire domain since the condition simulation of the finite-dimensional approximation of Gaussian processes is \textit{equivalent} to the simulation of a Gaussian vector restricted to convex sets. Finally, in \cite{golchi2015monotone}, the posterior mean is used as a predictive estimator whereas two estimators are computed by the methodology described in this paper (inequality mean and mode of the posterior distribution). Moreover, the last one (inequality mode) can be seen as the \textit{constrained interpolation function}, and then generalizes the correspondence established by Kimeldorf and Wahba \cite{KW1970} for \textit{constrained} interpolation (see \cite{2016arXiv160202714B} and \cite{Maatoukthesis2015}).

\begin{figure}[hptb]
\begin{minipage}{.5\linewidth}
\centering
\subfloat[]{\label{krigingGolchi}\includegraphics[scale=.3]{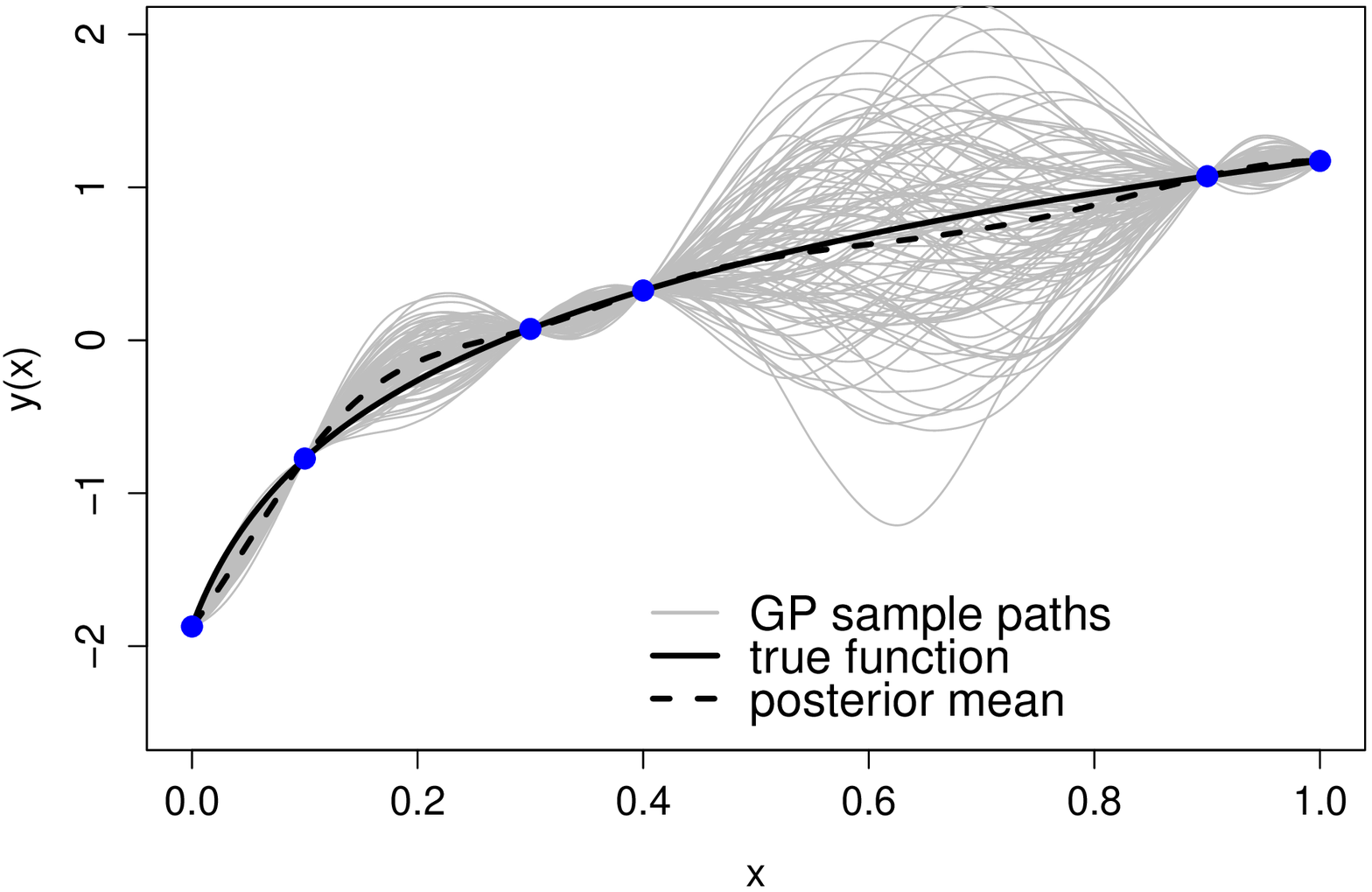}}
\end{minipage}%
\begin{minipage}{.5\linewidth}
\centering
\subfloat[]{\label{krigingGolchi1}\includegraphics[scale=.3]{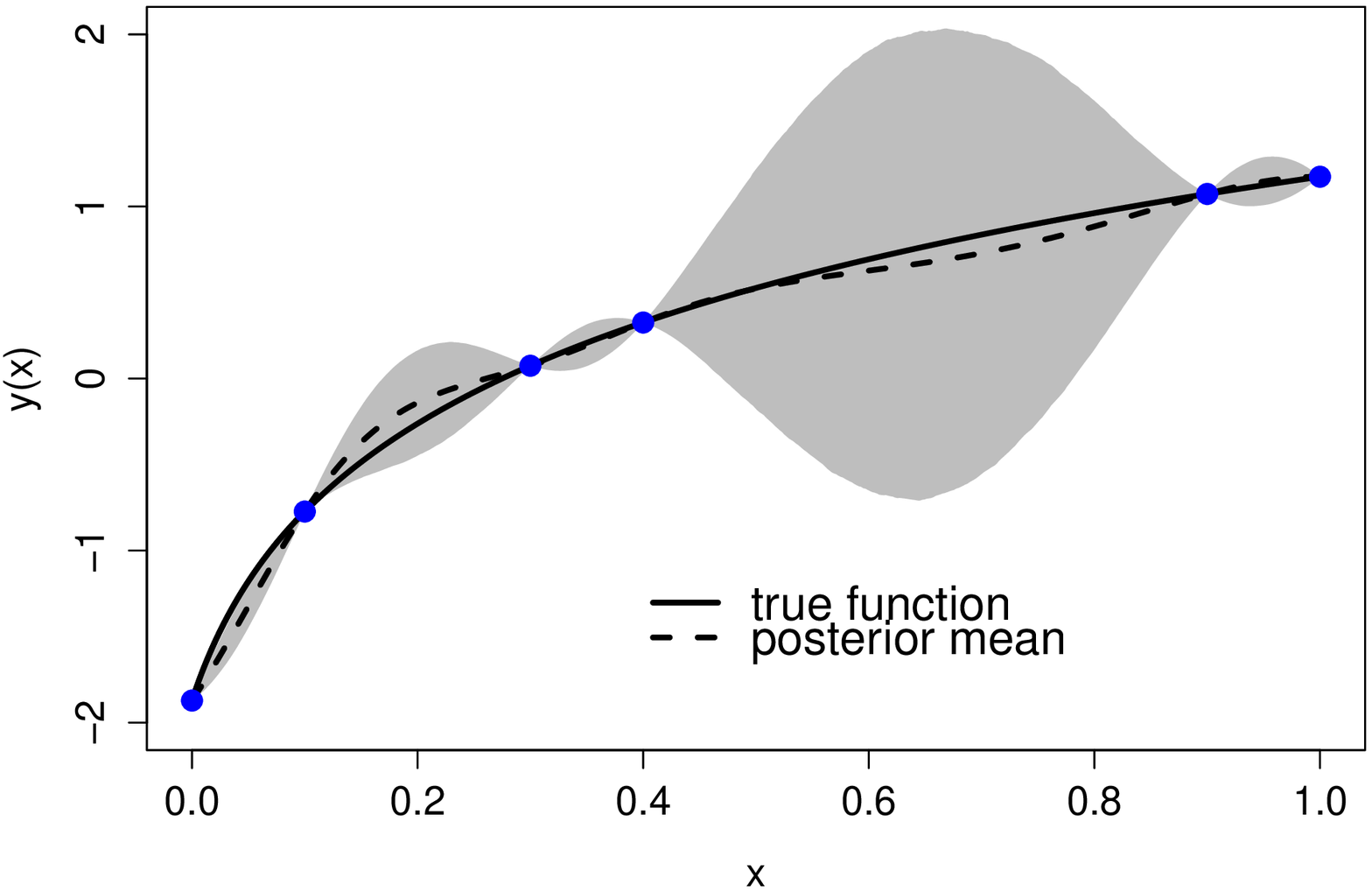}}
\end{minipage}\par\medskip
\centering
\subfloat[]{\label{Golchi1D1}\includegraphics[scale=.3]{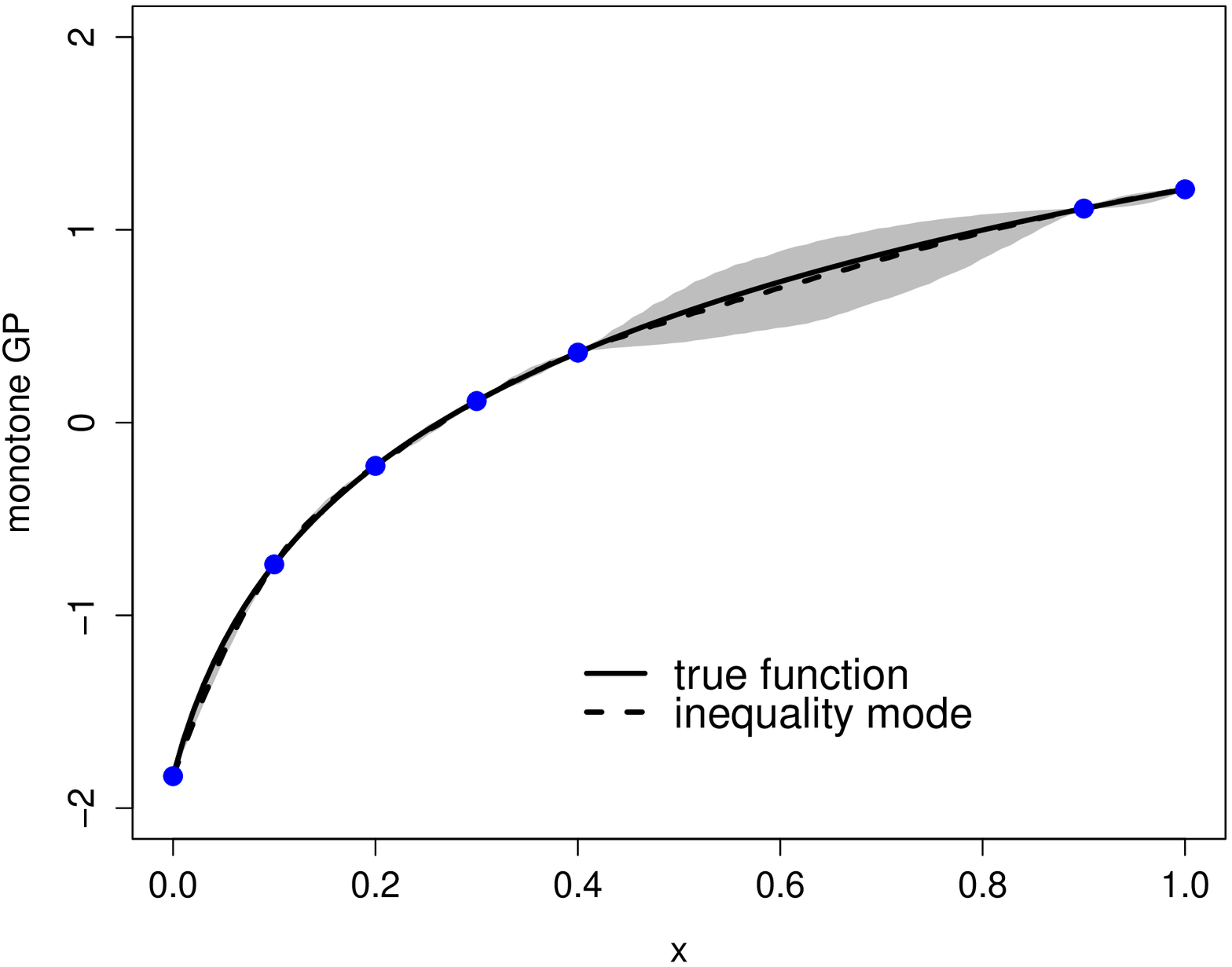}}
\caption{The real function evaluated at seven design points used in \cite{golchi2015monotone}~: (a) 100 sample paths taken from unconstrained GP with posterior mean, (b) posterior mean and 95\% prediction intervals from unconstrained GP, and (c) the inequality mode and 95\% prediction intervals from model \eqref{monotonicityapproach} conditionally to both interpolation conditions and monotonicity constraints.}
\label{GolchiExample}
\end{figure}

\section{Illustrative examples}\label{illustrative}
The aim of this section is to illustrate the proposed method with certain constraints such as boundedness, monotonicity and convexity and to show the difference between prediction functions (unconstrained Kriging mean, inequality Kriging mean and inequality mode). The simulation results are obtained by using Gaussian and Mat\'ern 3/2 covariance functions, where the \textit{constrained} evaluations are not taken from \textit{constrained} functions. We consider first one-dimension monotonicity, boundedness and convexity constraints examples. In two dimensions, we consider the monotonicity (non-decreasing) case with respect to the two input variables and to only one variable.

\subsection{Monotonicity in one dimensional case}
We begin with two monotonicity examples in one-dimension (Figure~\ref{dim1}). In Figure~\ref{dim12}, the 11 design points are given by $X=(0,0.05, 0.1, 0.3, 0.4,0.45, 0.5, 0.8, 0.85, 0.9, 1)$ and the corresponding output $\boldsymbol{y}=(0,0.6, 1.1, 5.5,\\
7.2, 8, 9.1, 15, 16.3, 17, 20)$. We choose $N=50$ and generate 40 sample paths taken from model \eqref{monotonicityapproach} conditionally to given data and monotonicity (non-decreasing) constraints $(Y'(u_j)\geq 0, \ j=0,\ldots,N)$. The Gaussian covariance function is used with the hyper-parameters $(\sigma^2,\theta)$ fixed to $(20^2,0.14)$. Notice that the simulated paths (gray lines) are non-decreasing in the entire domain, as well as the increasing Kriging mean $m^N_{\text{IK}}$ (solid line). The usual (unconstrained) Kriging mean $m^N_{\text{K}}$ and the inequality mode $M^N_{\text{IK}}$ (dash-dotted line) coincide and are also non-decreasing. This is because $\boldsymbol{\xi}_{\text{I}}$ is inside the acceptance region $C_{\boldsymbol{\xi}}$. In Figure~\ref{dim11}, the input is $X=(0,0.3,0.4,0.5, 0.9)$ and the corresponding output is $\boldsymbol{y}=(0, 4,6,6.6, 10)$. Again, the Gaussian covariance is used with the parameters $(\sigma^2,\theta)$ fixed to $(20^2,0.25)$. The increasing Kriging mean (solid line) and the inequality mode satisfy monotonicity (non-decreasing) constraints, contrarily to the usual (unconstrained) Kriging mean (dash-dotted line)~: it corresponds to the situation where $\boldsymbol{\xi}_{\text{I}}$ lies outside the acceptance region $C_{\boldsymbol{\xi}}$.

\begin{figure}[hptb]
\begin{minipage}{.5\linewidth}
\centering
\subfloat[]{\label{dim12}\includegraphics[scale=.3]{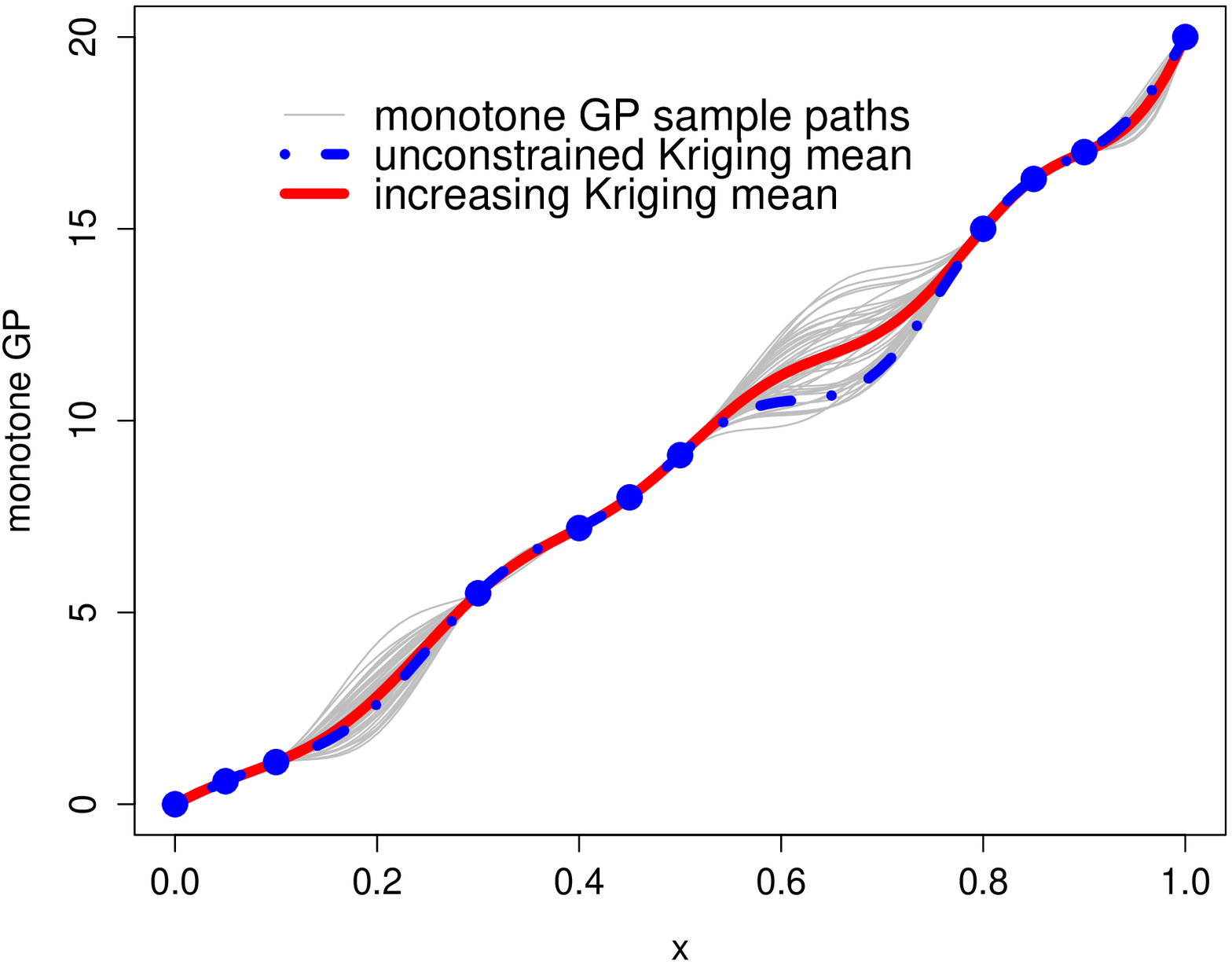}}
\end{minipage}%
\begin{minipage}{.5\linewidth}
\centering
\subfloat[]{\label{dim11}\includegraphics[scale=.3]{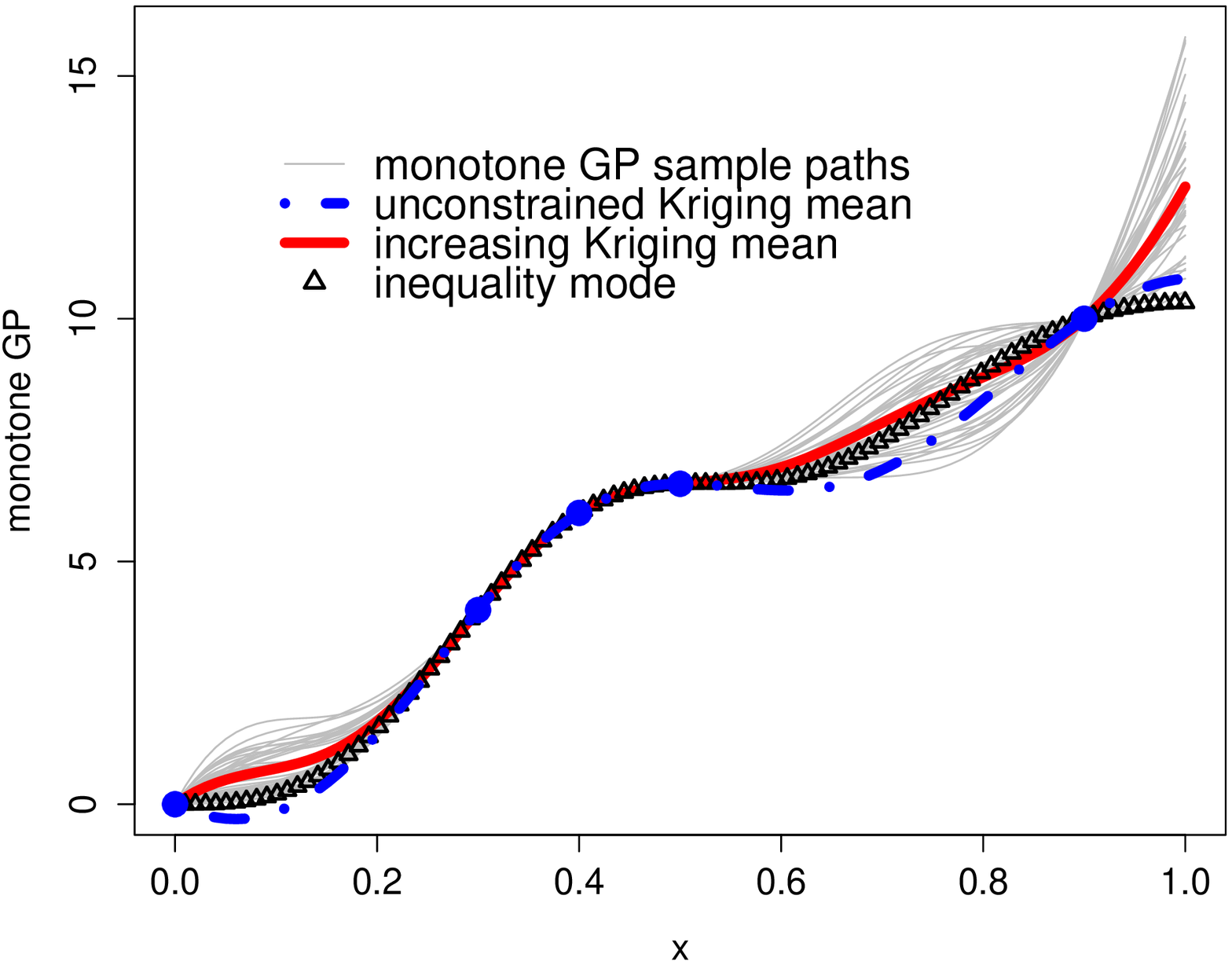}}
\end{minipage}
\caption{Simulated paths drawn from model \eqref{monotonicityapproach} respecting non-decreasing constraints in the entire domain.
The usual Kriging mean (dash-dotted line) coincides with the inequality mode and respects the monotonicity in Figure~\ref{dim12}, but not in Figure~\ref{dim11}.}
\label{dim1}
\end{figure}

\subsection{Monotonicity of continuous but non-derivable functions}
The \textit{constrained} evaluations are given by the input vector $X=(0.1, 0.2, 0.3, 0.6, 0.9, 1)$ and the corresponding output $\boldsymbol{y}=(-1, 1, 2, 3, 4.5, 8)$ (Figure~\ref{simmon1DC0}). 
We choose $N=50$ then we have $51$ knots and we generate 40 sample paths drawn from the finite-dimensional approximation of GPs defined in \eqref{boundaryapproach} conditionally to data interpolation and monotonicity constraints given in Remark~\ref{mon1DC0}. The Mat\'ern 5/2 covariance function is used with the parameters fixed to $(\sigma^2,\theta)=(40^2,1.2)$. The sample paths (gray solid lines) are continuous (non-derivable) and non-decreasing in the entire domain, contrarily to the usual (unconstrained) Kriging mean. The inequality mode (maximum \textit{a posteriori}) and the increasing Kriging mean (mean \textit{a posteriori}) verify monotonicity (non-decreasing) constraints in the entire domain. It corresponds to the situation where $\boldsymbol{\xi}_{\text{I}}$ lies outside the acceptance region $C_{\boldsymbol{\xi}}$.

\begin{figure}[hptb]
\centering
\includegraphics[scale=.4]{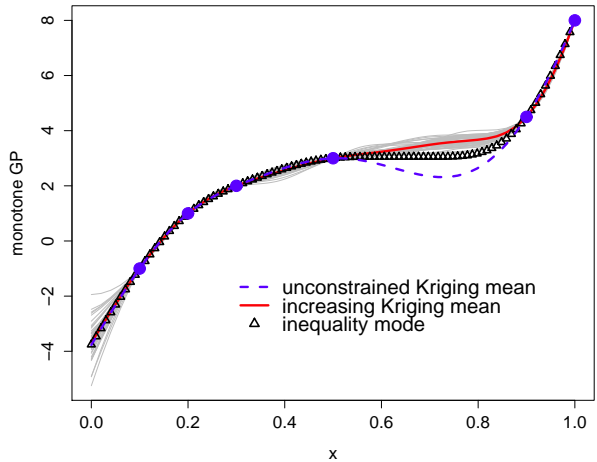}
\caption{Simulated paths drawn from model \eqref{boundaryapproach} using Remark~\ref{mon1DC0}.
Notice that the simulated paths are continuous (non-derivable) and non-decreasing in the entire domain.}
\label{simmon1DC0}
\end{figure}

\subsection{Boundedness constraints in one dimensional case}
Now, we consider the positive and boundedness constraints (Figure~\ref{posbond}). We choose $N=50$ and generate 100 sample paths taken from the finite-dimensional approximation defined in \eqref{boundaryapproach} conditionally to given data and boundedness constraints. In both figures, the Gaussian covariance function is used with the parameters $(\sigma^2,\theta)=(4^2,0.13)$ (Figure~\ref{positive}) and $(\sigma^2,\theta)=(25^2,0.2)$ (Figure~\ref{bornee}). In Figure~\ref{positive}, $\boldsymbol{\xi}_{\text{I}}$ is inside the acceptance region and the usual (unconstrained) Kriging mean coincides with the inequality mode and respects boundedness constraints, contrarily to Figure~\ref{bornee}, where $\boldsymbol{\xi}_{\text{I}}$ lies outside the acceptance region. Notice that the simulated paths satisfy the inequality constraints in the entire domain (nonnegative (Figure~\ref{positive})) and are bounded between -20 and 20 (Figure~\ref{bornee}). From Figure~\ref{bornee}, one can remark that the degree of smoothness of the inequality mode is related to one of the covariance function $K$ of the original GP, see Remark~\ref{correspondRemark}.

\begin{figure}[hptb]
\begin{minipage}{.5\linewidth}
\centering
\subfloat[]{\label{positive}\includegraphics[scale=.3]{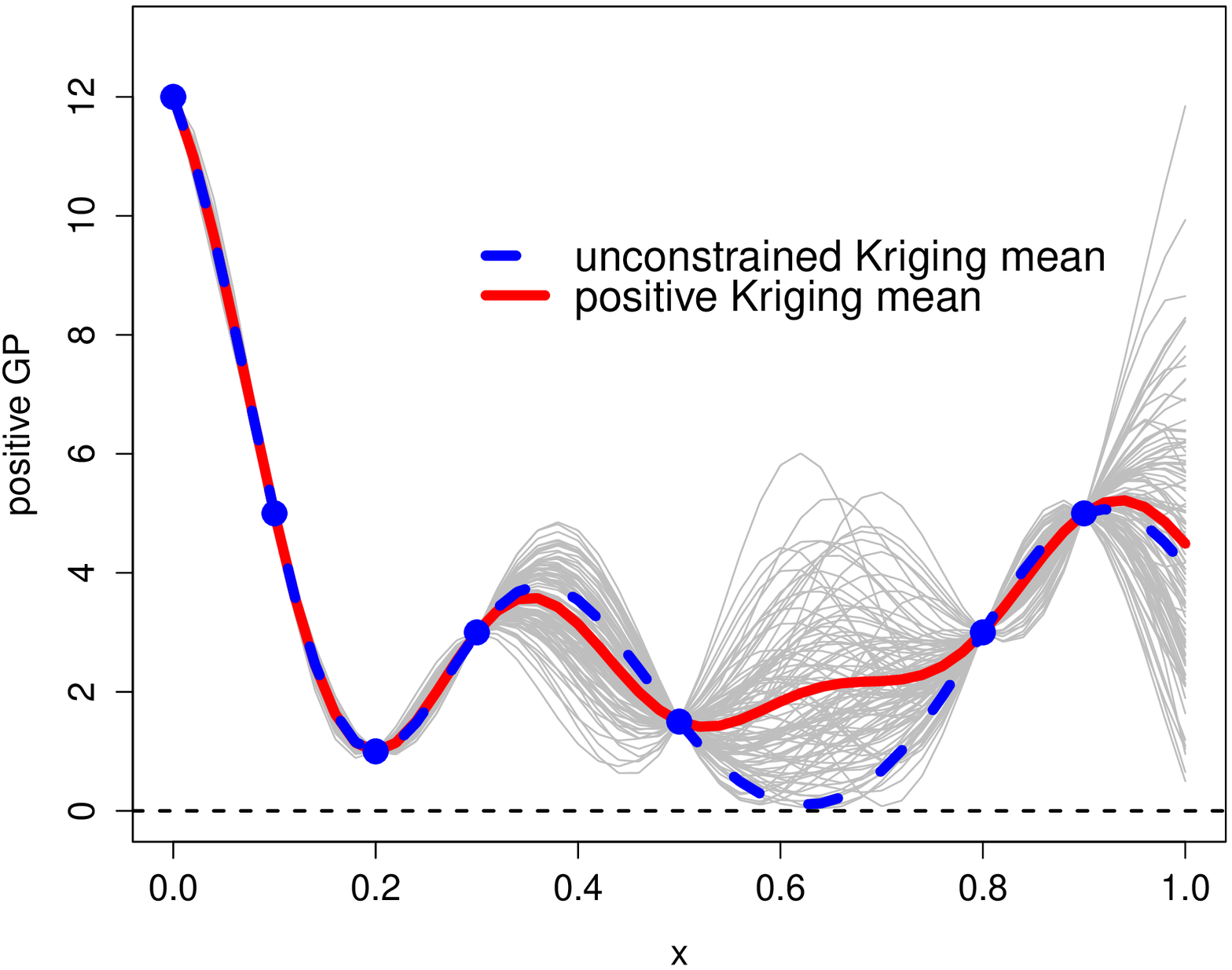}}
\end{minipage}%
\begin{minipage}{.5\linewidth}
\centering
\subfloat[]{\label{bornee}\includegraphics[scale=.3]{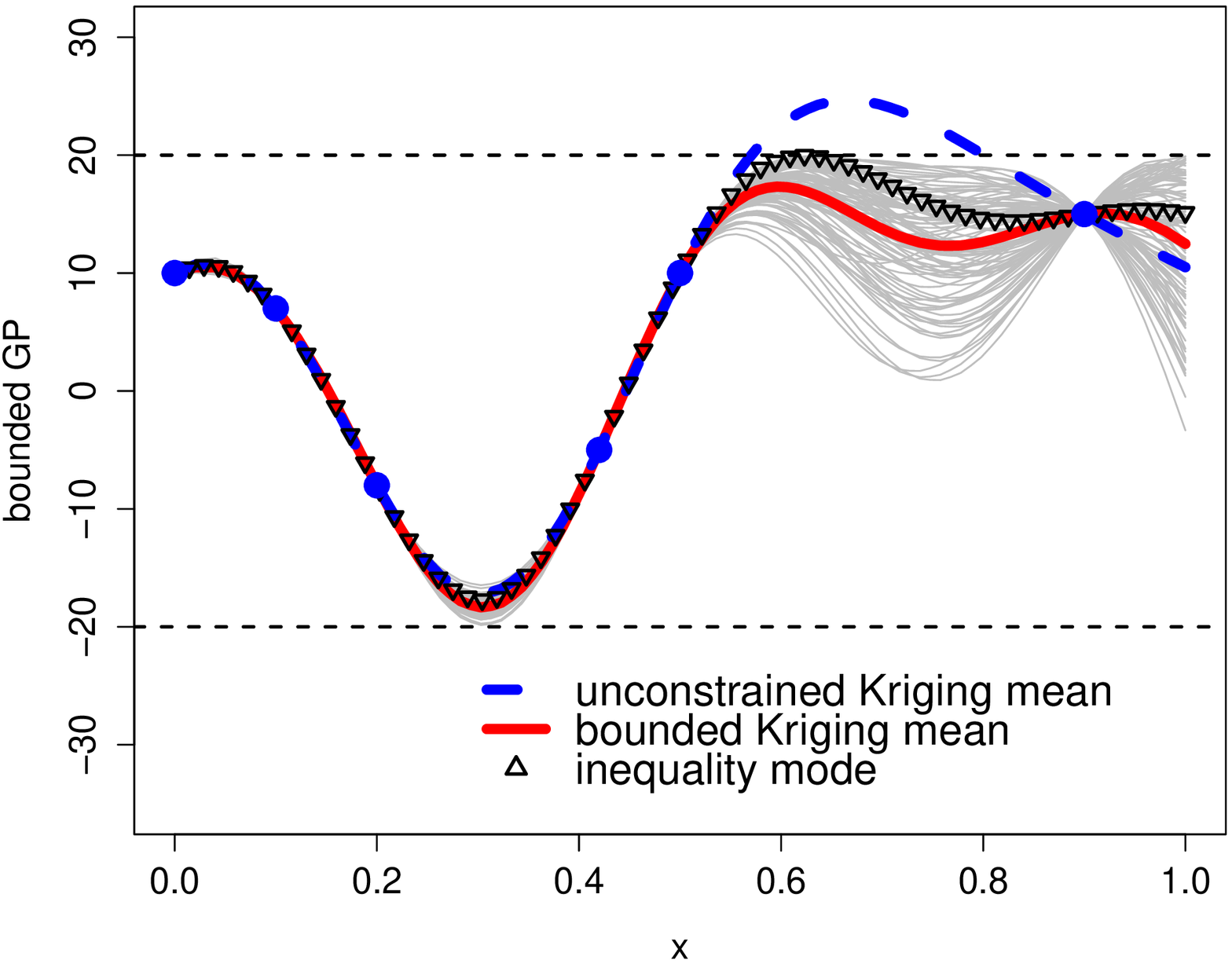}}
\end{minipage}
\caption{Simulated paths drawn from model \eqref{boundaryapproach} respecting positivity constraints (Figure~\ref{positive}) and boundedness constraints (Figure~\ref{bornee}). The usual (unconstrained) Kriging mean and the inequality mode coincide in Figure~\ref{positive}, but not in Figure~\ref{bornee}.}
\label{posbond}
\end{figure}

\subsection{Convexity constraints in one dimensional case}
The \textit{constrained} evaluations in Figure~\ref{convex} are given by $X=(0,0.05,0.2,0.5,0.85,0.95)$ and the corresponding output $\boldsymbol{y}=(20, 15, 3, -5, 7, 15)$. We choose $N=50$ and generate 25 sample paths taken from model \eqref{convexityapproach} conditionally to given data and convexity constraints $(\xi_j\geq 0, \ j=0,\ldots,N)$. The Gaussian covariance function is used with the parameters fixed to $(\sigma^2,\theta)=(10^2,0.2)$. The simulated paths, the inequality mode (maximum \textit{a posteriori}) and the convex Kriging mean (mean \textit{a posteriori}) are convex in the entire domain, contrarily to the usual (unconstrained) Kriging mean (dash-dotted line). It corresponds to the situation where $\boldsymbol{\xi}_{\text{I}}$ lies outside the acceptance region $C_{\boldsymbol{\xi}}$.

\begin{figure}[hptb]
\begin{minipage}{.5\linewidth}
\centering
\subfloat[]{\label{convex6pts0}\includegraphics[scale=.3]{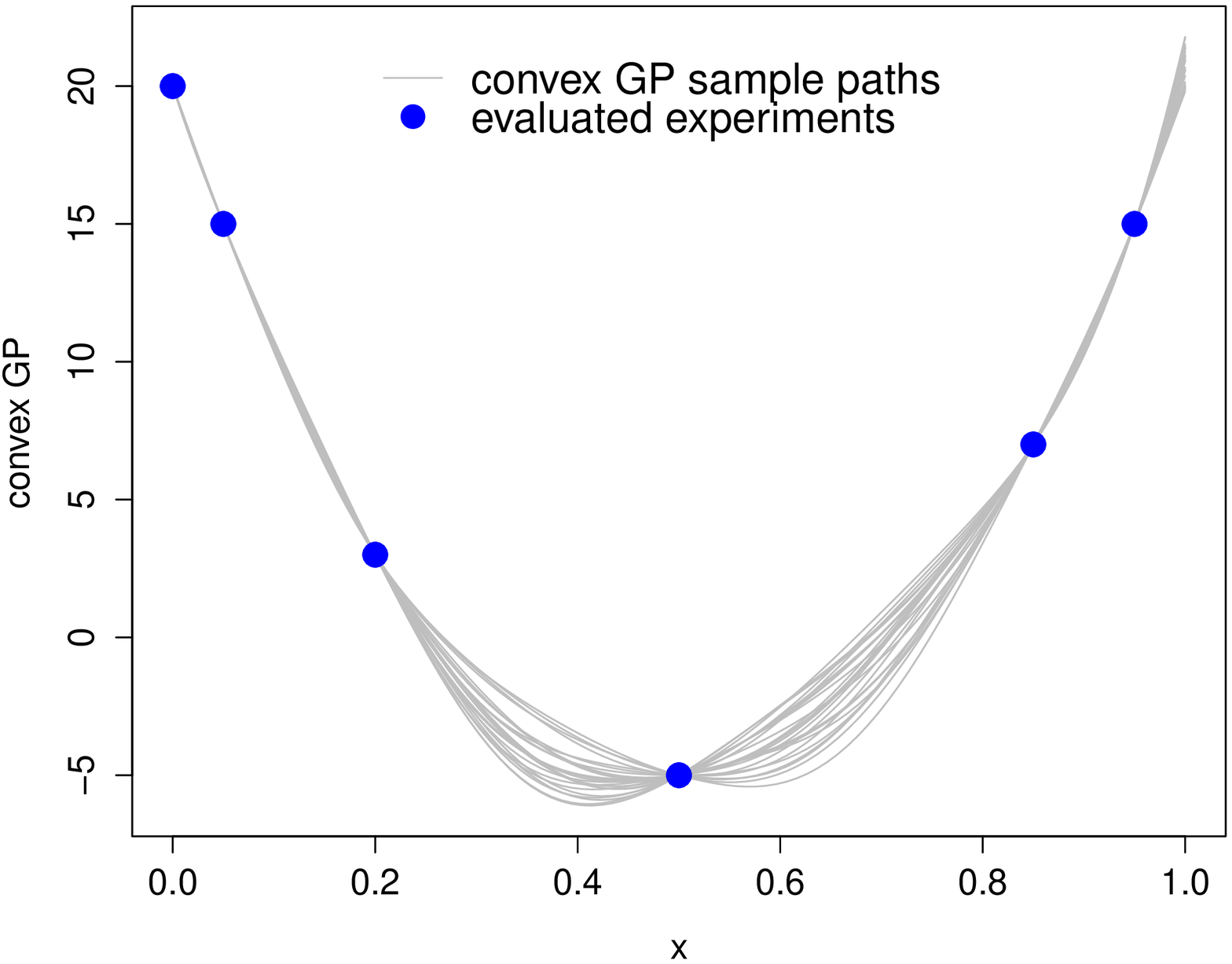}}
\end{minipage}%
\begin{minipage}{.5\linewidth}
\centering
\subfloat[]{\label{convex6pts}\includegraphics[scale=.3]{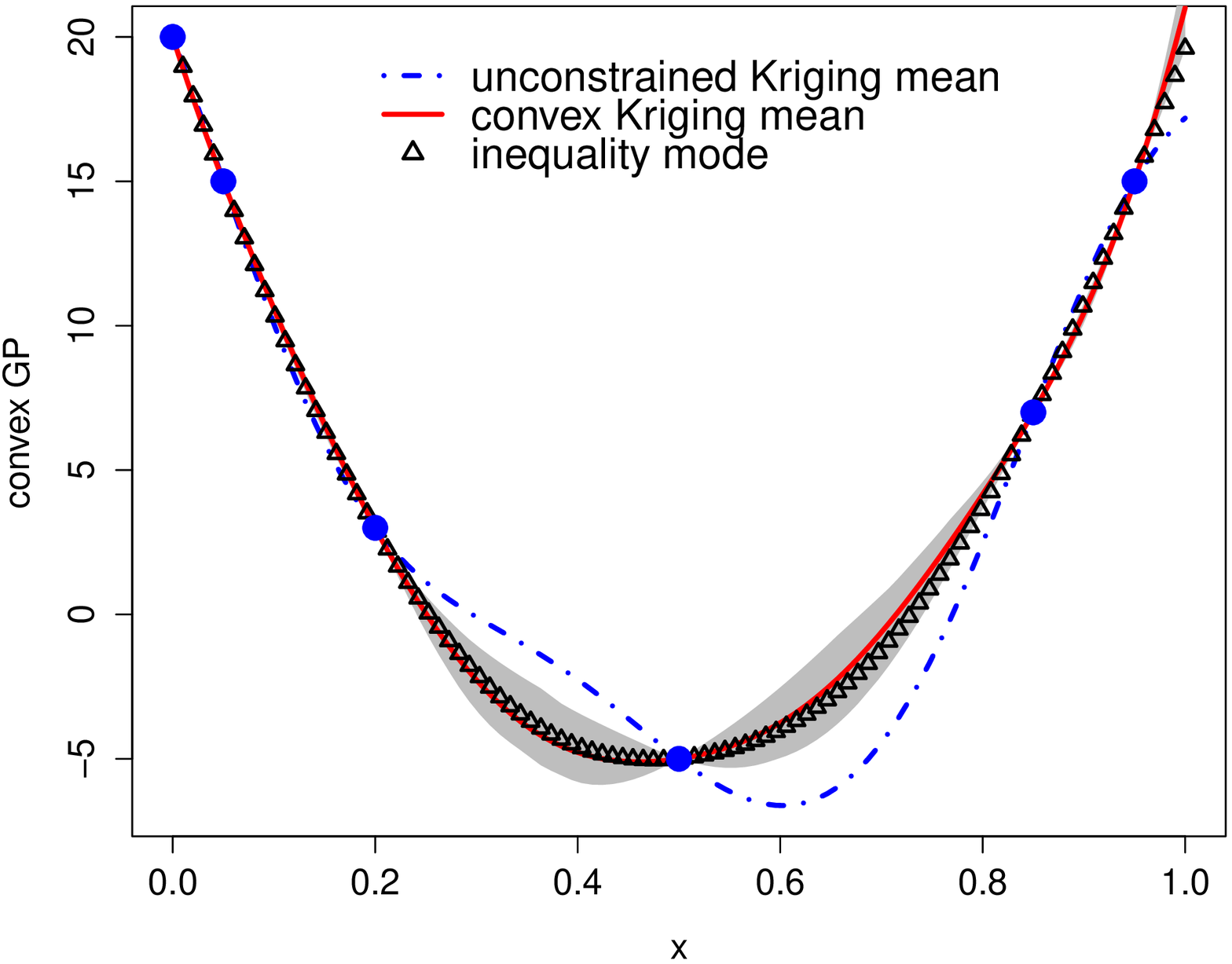}}
\end{minipage}
\caption{Simulated paths drawn from model \eqref{convexityapproach} respecting convexity constraints in the entire domain (Figure~\ref{convex6pts0}). 95\% prediction intervals together with three estimators~: unconstrained Kriging mean, inequality mode and convex Kriging mean (Figure~\ref{convex6pts}).}
\label{convex}
\end{figure}

\subsection{Isotonicity in two dimensions}
In two dimensions, the aim is to interpolate a 2D-function defined on $[0,1]^2$ and non-decreasing with respect to the two inputs. In that case, and by the uniform subdivision of the input set the number of knots and basis functions is $(N+1)^2$. In Figures~\ref{grid}, \ref{dim2} and \ref{monotonesurfaces},
we choose $N=7$, then we have $64$ knots and basis functions. Suppose that the real function is evaluated at four design points given by the rows of the $4\times 2$ matrix 
$\boldsymbol{X}=\left[
\begin{matrix}
	0.1 && 0.9 && 0.5 && 0.8\\
	0.4 && 0.3 && 0.6 && 0.9\\
\end{matrix}\right]^\top$ and the corresponding output $\boldsymbol{y}=\left(5,12,13,25\right)$.
The output values respect monotonicity (non-decreasing) constraints in two dimensions. The two-dimensional Gaussian kernel is used
\begin{equation*}
K(\boldsymbol{x},\boldsymbol{x}')=\sigma^2\exp\left(-\frac{(x_1-x_1')^2}{2\theta_1^2}\right)\times\exp\left(-\frac{(x_2-x_2')^2}{2\theta_2^2}\right),
\end{equation*}
where the variance parameter $\sigma$ is fixed to 10 and the length parameters $(\theta_1,\theta_2)$ to $(1,1)$.
We generate $5$ simulation surfaces taken from model \eqref{monotonicity2Dapproach} conditionally to given data and monotonicity (non-decreasing) constraints with respect to the two input variables (Figure~\ref{4pdim2}). The two red surfaces are the 95\% prediction interval. To check visually the isotonicity, we plot in Figure~\ref{contour} the contour levels of one simulation surface. The blue points represent the interpolation input locations (design points). If we fix one of the variables and we draw the vertical or horizontal line, it must not intersect a contour level two times. \\

\begin{figure}[hptb]
\centering
\includegraphics[scale=0.4]{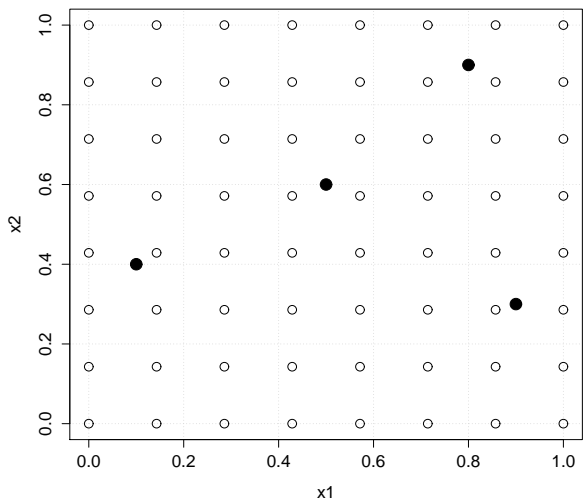}
\caption{Design points for the monotone 2D interpolation problem (black points)
and knots $(u_i,u_j)_{0\leq i,j\leq 7}$ used to define the basis functions.}
\label{grid}
\end{figure}

\begin{figure}[hptb]
\begin{minipage}{.5\linewidth}
\centering
\subfloat[]{\label{4pdim2}\includegraphics[scale=.4]{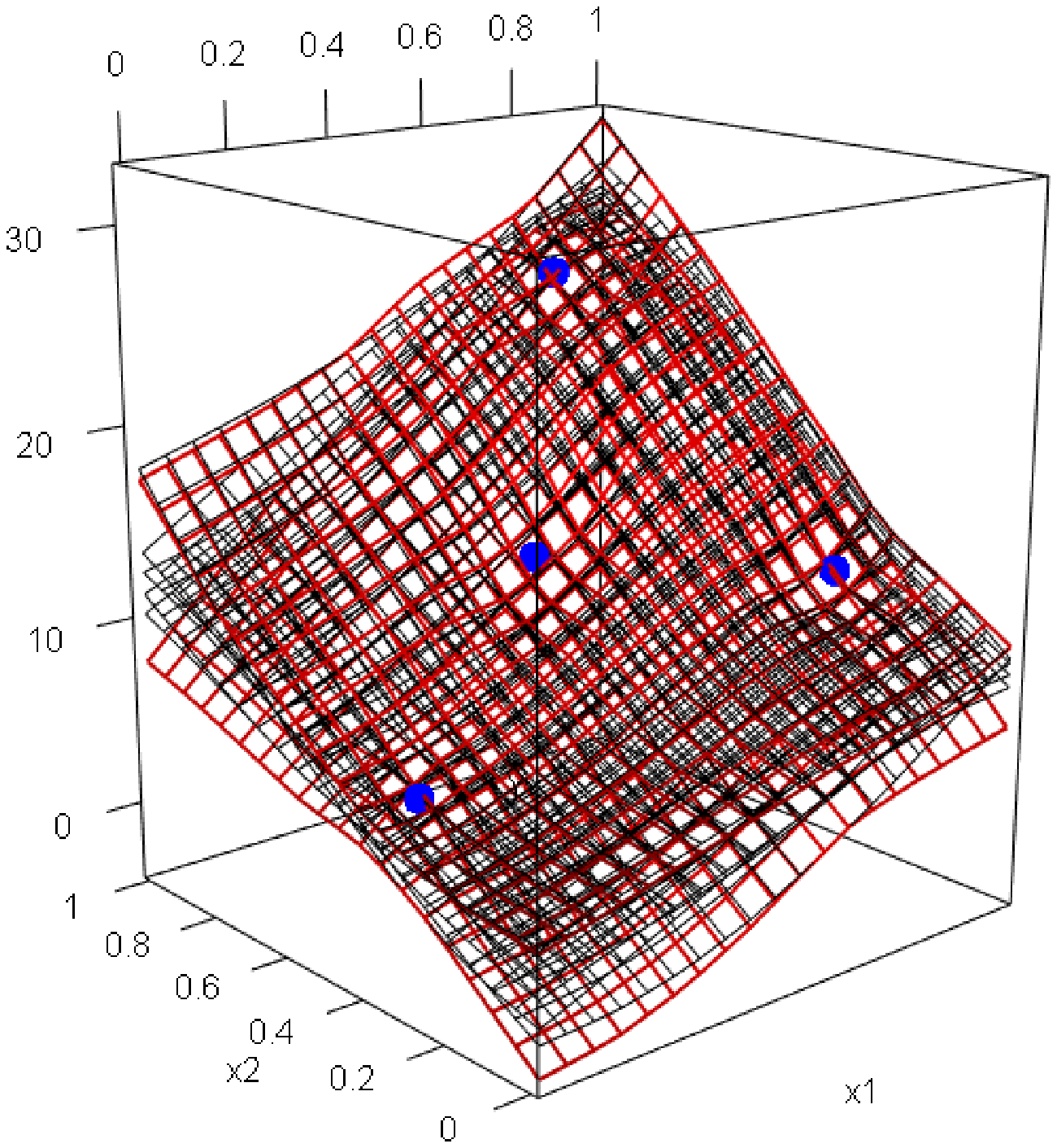}}
\end{minipage}%
\begin{minipage}{.5\linewidth}
\centering
\subfloat[]{\label{contour}\includegraphics[scale=.5]{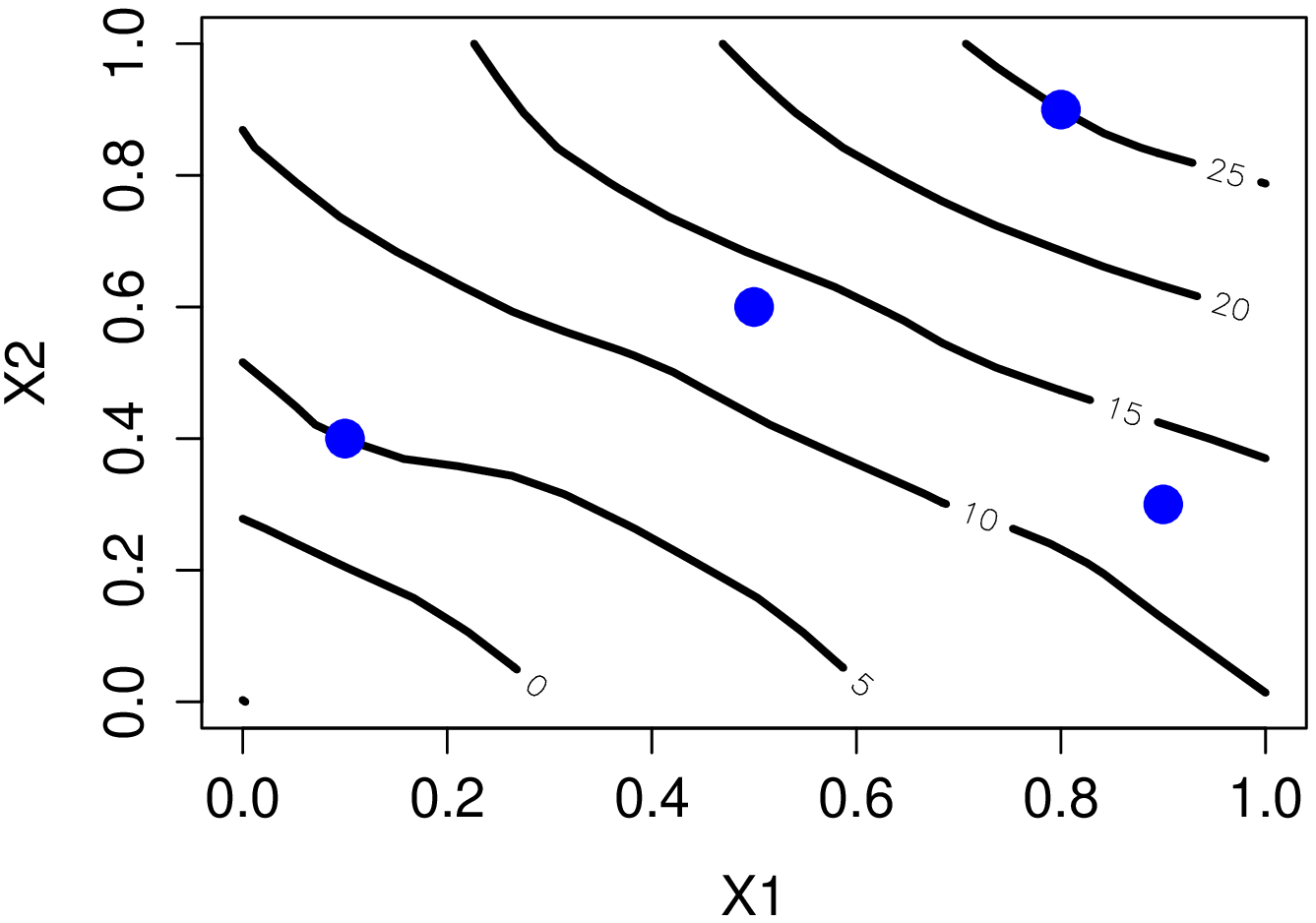}}
\end{minipage}
\caption{Simulated surfaces drawn from model \eqref{monotonicity2Dapproach} respecting monotonicity constraints for the two input variables Figure~\ref{4pdim2}. Contour levels for one
simulated surface Figure~\ref{contour}.}
\label{dim2}
\end{figure}

\begin{figure}[hptb]
  \centering
  \subfloat[][\label{surfacemonotone1}]{\includegraphics[width=.4\textwidth]{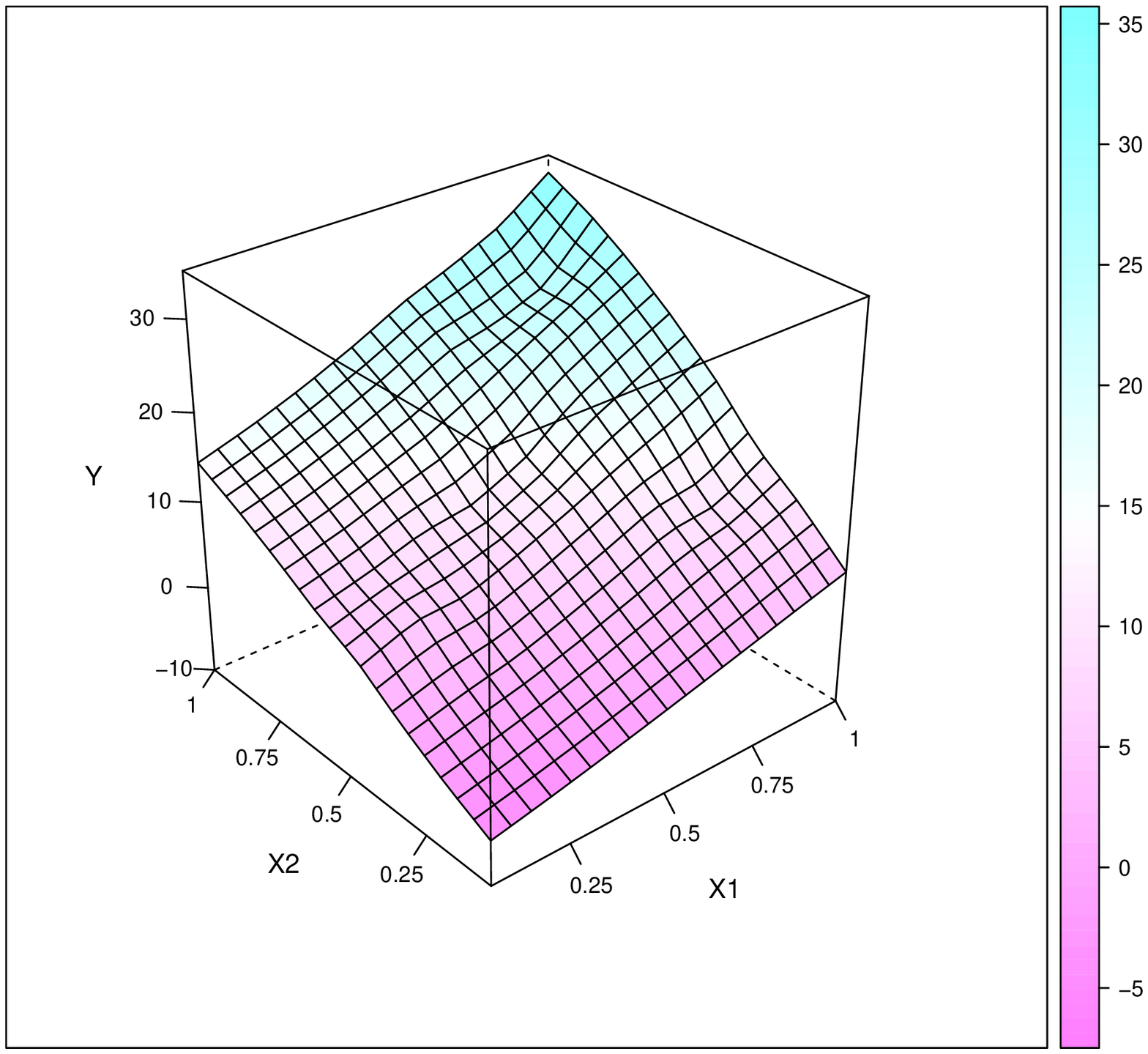}}\quad
  \subfloat[][\label{surfacemonotone2}]{\includegraphics[width=.4\textwidth]{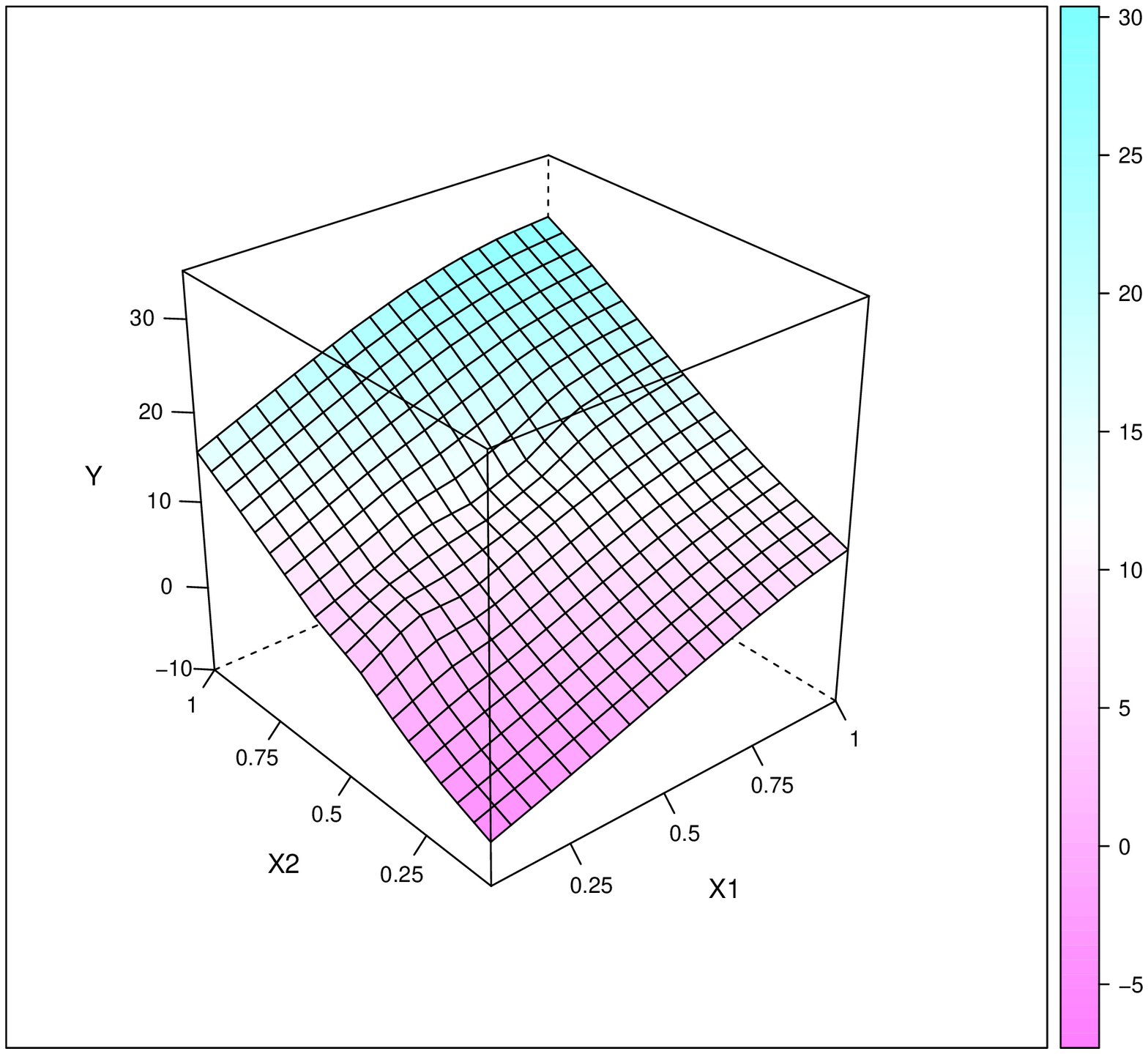}}\\
  \subfloat[][\label{surfacemonotone3}]{\includegraphics[width=.4\textwidth]{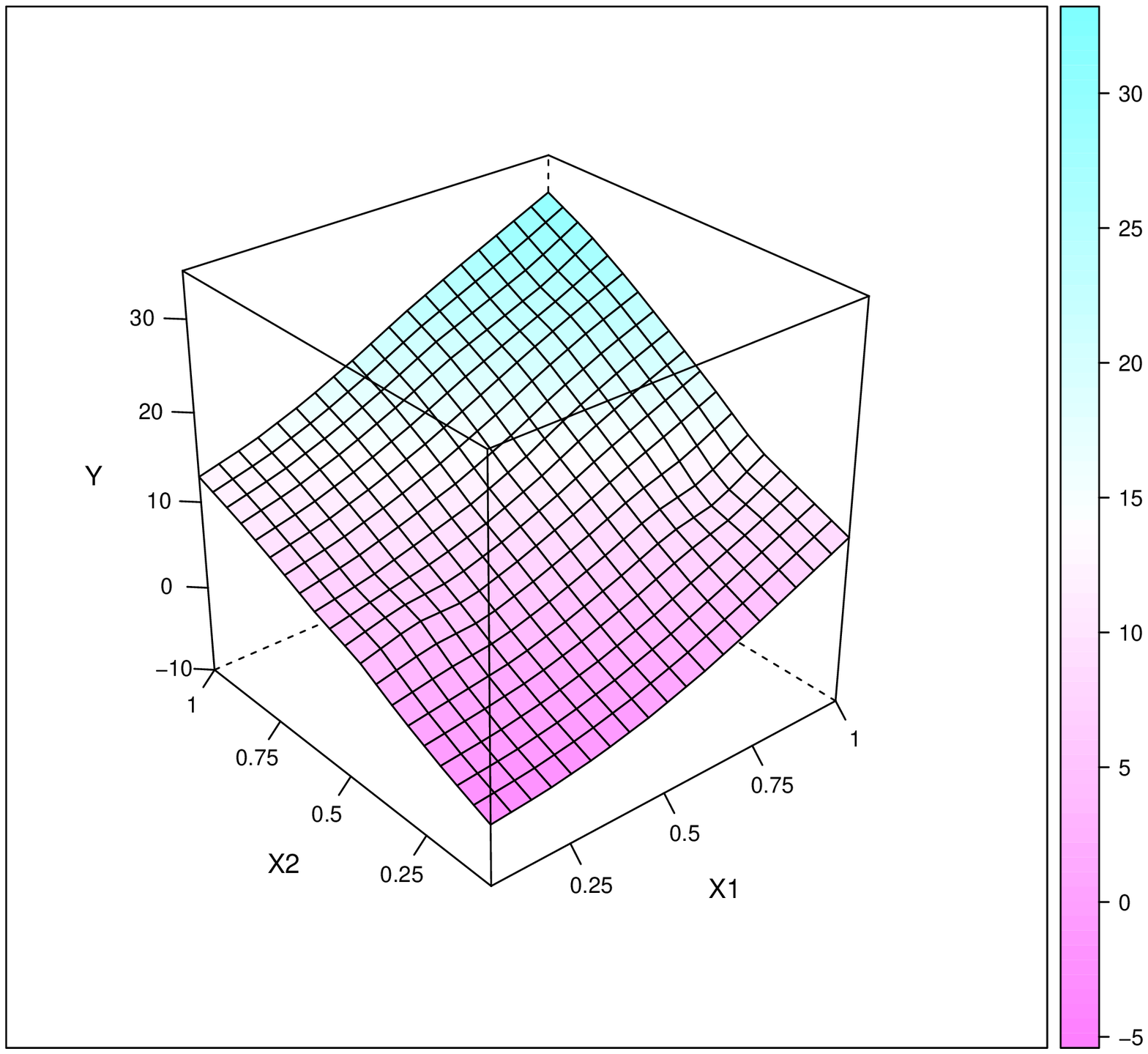}}\quad
  \subfloat[][\label{surfacemonotone4}]{\includegraphics[width=.4\textwidth]{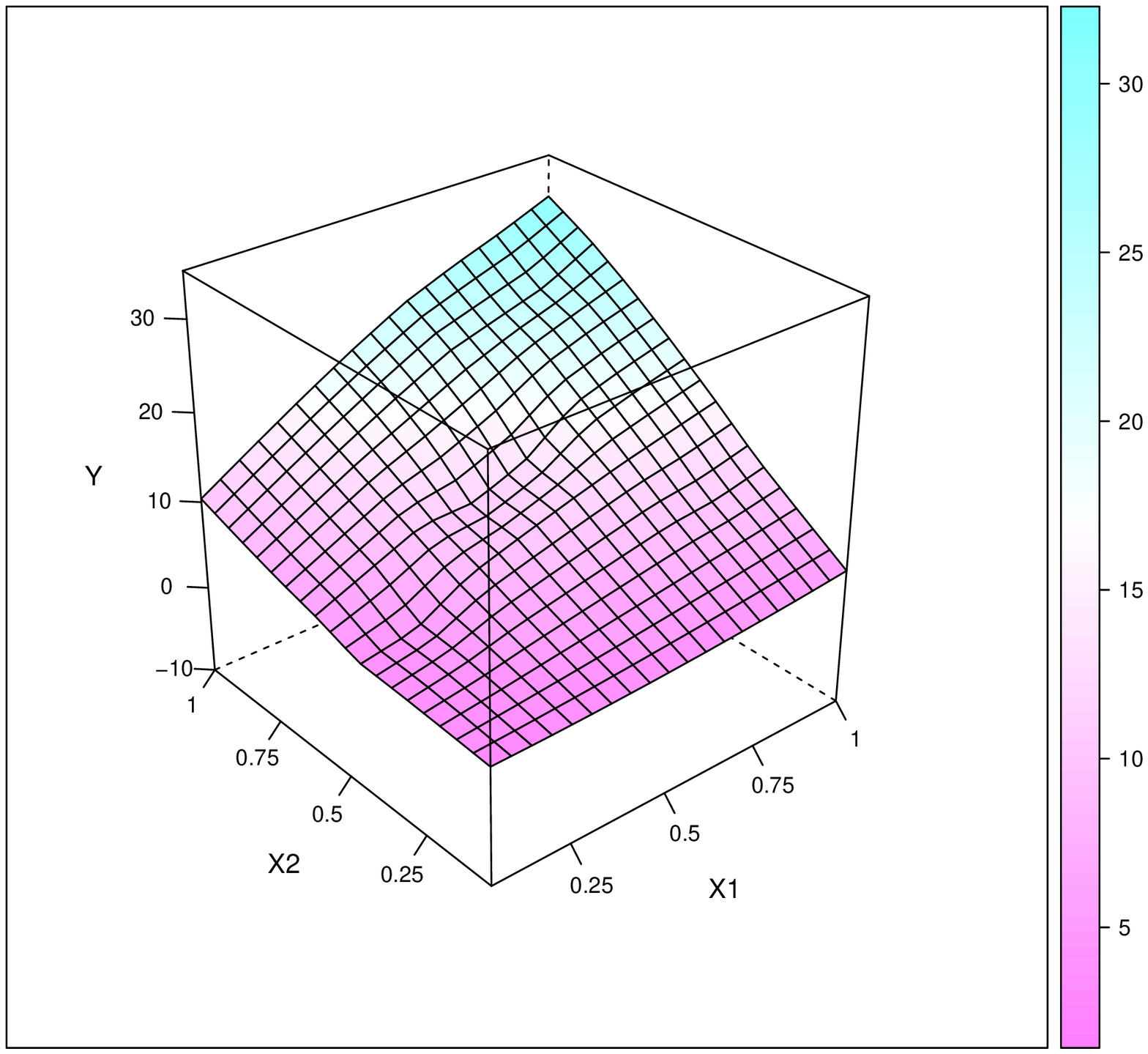}}
  \caption{Simulated surfaces drawn from the example used in Figure~\ref{4pdim2} respecting monotonicity constraints for the two input variables.}
  \label{monotonesurfaces}
\end{figure}

In Figure~\ref{monotonesurfaces}, we draw some simulated surfaces taken from the example used in Figure~\ref{4pdim2}. All the simulated surfaces are non-decreasing with respect to the two input variables.\\

\begin{figure}[hptb]
\begin{minipage}{.5\linewidth}
\centering
\subfloat[]{\label{monotoneX12D}\includegraphics[scale=.26]{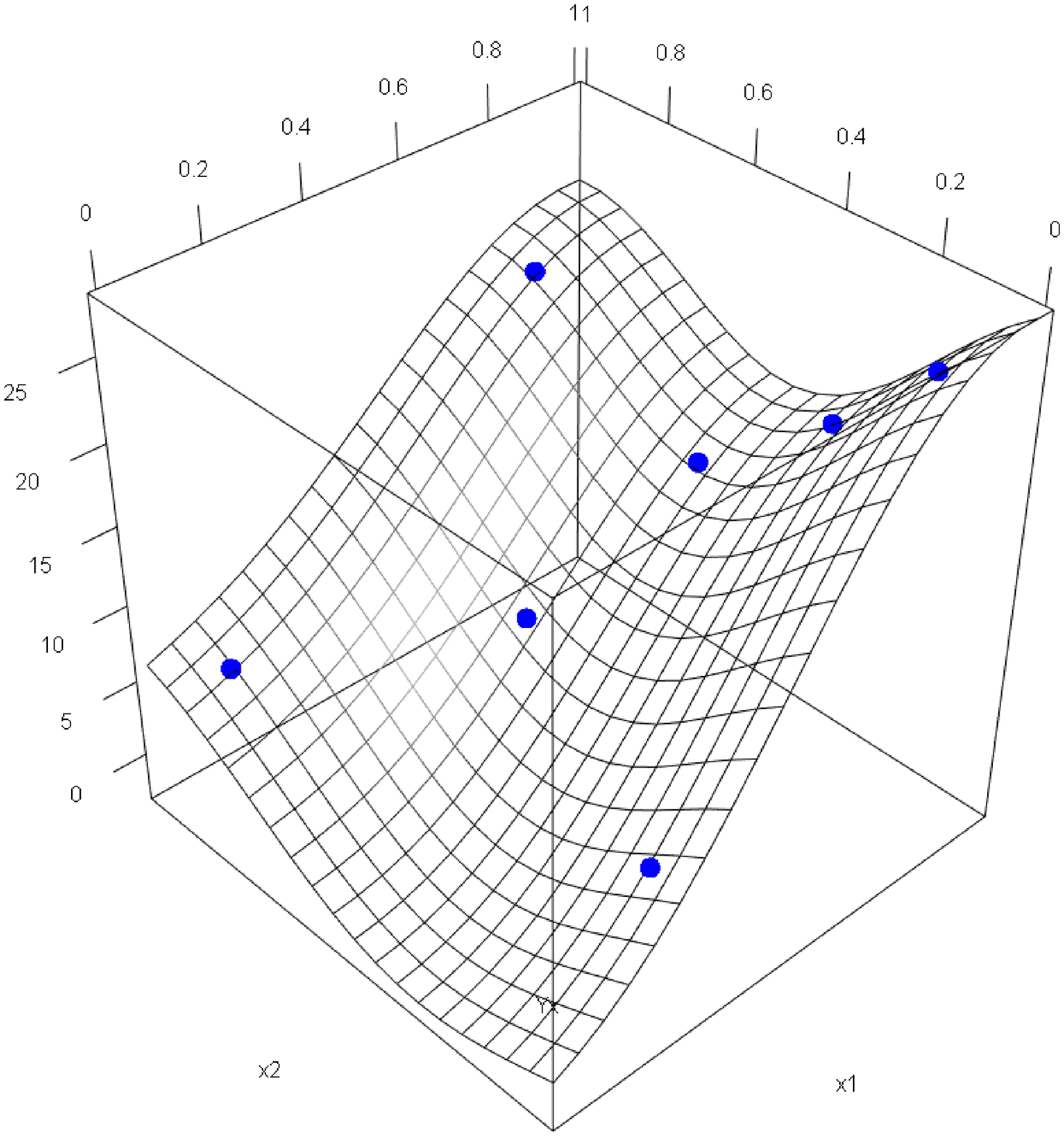}}
\end{minipage}%
\begin{minipage}{.5\linewidth}
\centering
\subfloat[]{\label{contourmonotoneX1}\includegraphics[scale=.31]{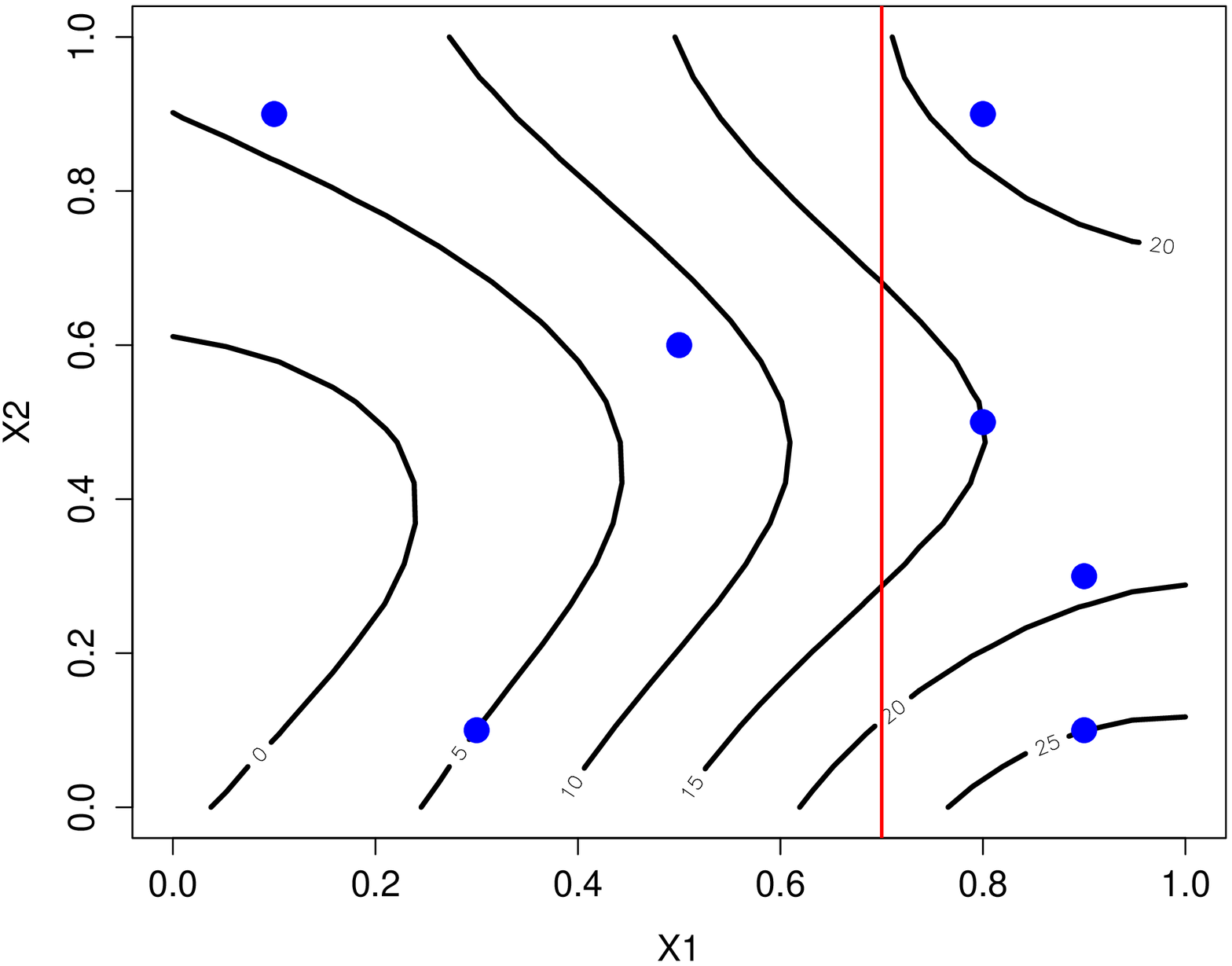}}
\end{minipage}
\caption{Simulated surface drawn from model \eqref{monotonicity2Dx1} respecting monotonicity (non-decreasing) constraints for the first variable only, and the associated contour levels.}
\label{dim2onevariablecontour}
\end{figure}

In Figure~\ref{dim2onevariablecontour}, a simulation surface of the conditional GP at four design points including monotonicity (non-decreasing) constraints with respect to the first input variable only is shown. In that case, we choose $N=23$, so we have $(N+1)^2$ basis functions and knots. The two-dimensional Gaussian covariance function is used with the variance parameter $\sigma^2$ fixed to $10^2$ and the length hyper-parameters $(\theta_1,\theta_2)$ fixed to $(0.5,0.45)$.\\

\section{Numerical convergence}
In order to investigate the convergence rate of the proposed model when $N$ tends to infinity, we plot in Figure~\ref{krigingMode} the inequality mode and the usual (unconstrained) Kriging mean in the situation where they are different. It corresponds to the case where the usual Kriging mean does not respect boundedness constraints (i.e. $\boldsymbol{\xi}_{\text{I}}\notin C_{\boldsymbol{\xi}}$). In Figure~\ref{mode500N20}, we illustrate the inequality mode $M_{\text{IK}}^N$ of the finite-dimensional approximation defined in \eqref{boundaryapproach} when $N=500$. The dashed-line represents $M_{\text{IK}}^N$ when $N=20$, which is close to one generated from $N=500$. Let us specify that the Mat\'ern 3/2 covariance function is used with the length parameter $\theta=0.25$ (see Table~\ref{kernel}).

\begin{figure}[hptb]
\begin{minipage}{.5\linewidth}
\centering
\subfloat[]{\label{krigingMode}\includegraphics[scale=.3]{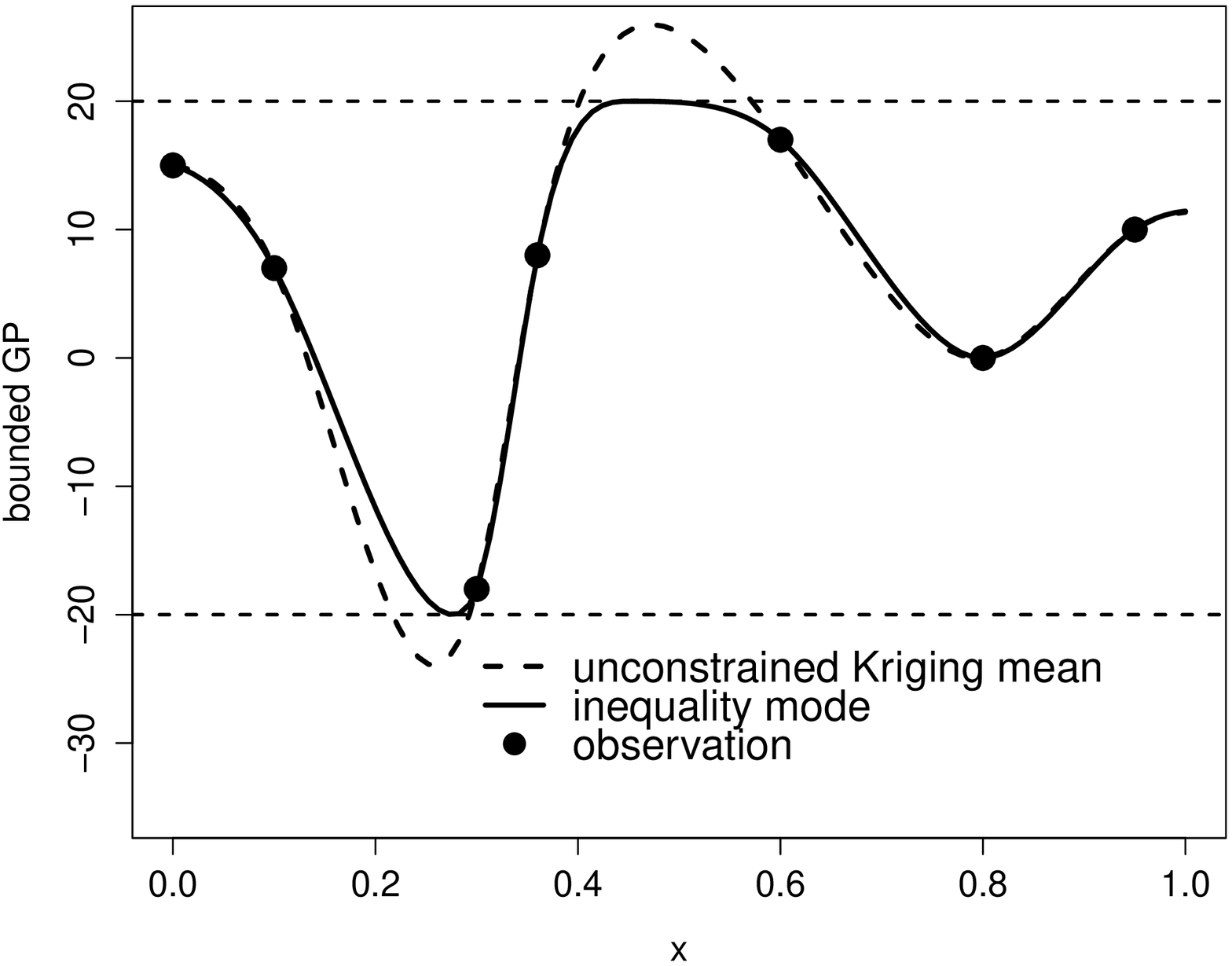}}
\end{minipage}%
\begin{minipage}{.5\linewidth}
\centering
\subfloat[]{\label{mode500N20}\includegraphics[scale=.3]{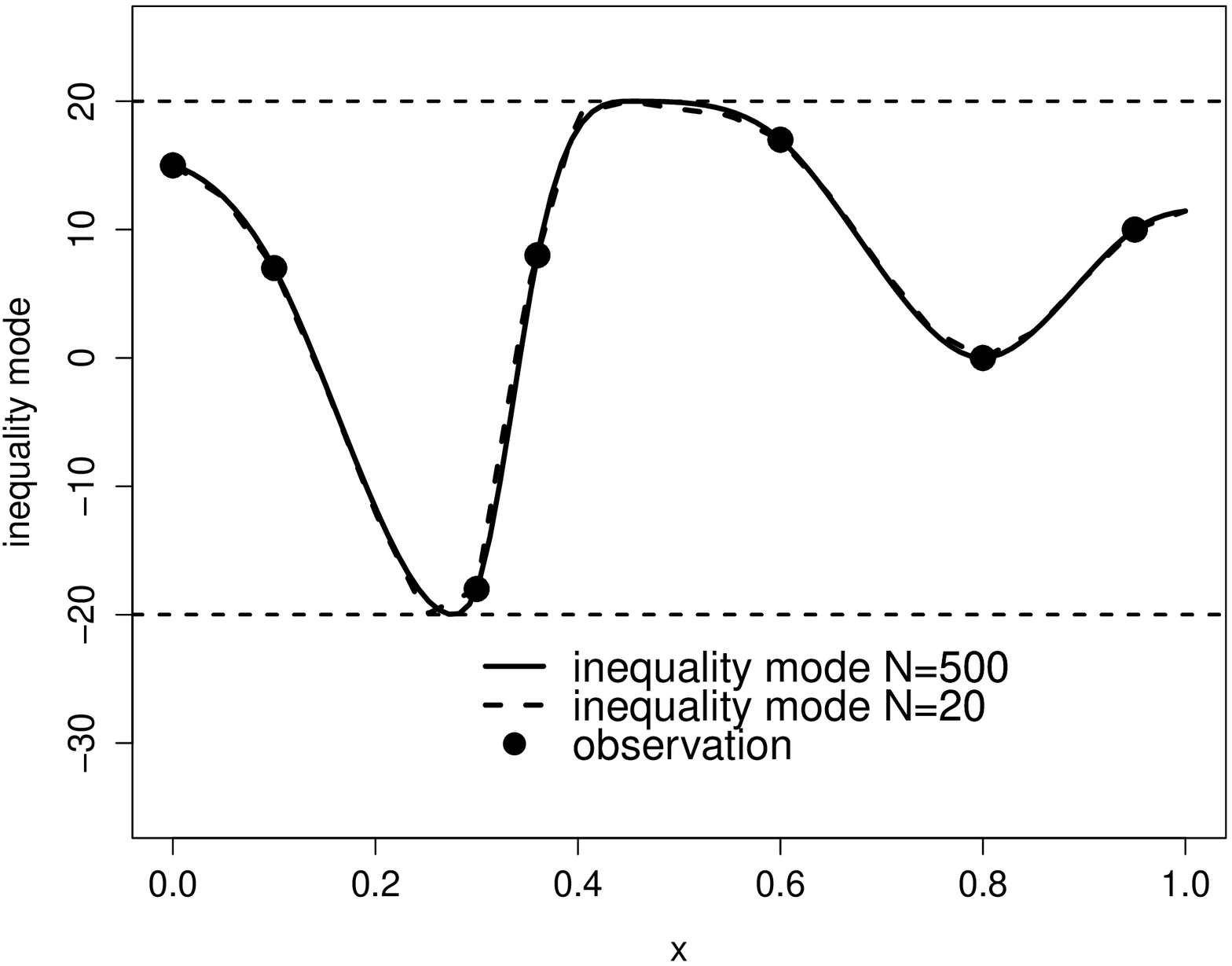}}
\end{minipage}
\caption{In both figures, the solid line represents the inequality mode $M^N_{\text{IK}}$ when $N=500$ which respects boundedness constraints in the entire domain. The dashed-line in Figure~\ref{krigingMode} (resp. Figure~\ref{mode500N20}) represents the usual Kriging mean (resp. the inequality mode when $N=20$).}
\label{approachmode}
\end{figure}

The convergence when $N$ tends to infinity of the finite-dimensional approximation defined in \eqref{monotonicityapproach} with monotonicity constraints is studied in Figure~\ref{approachmodemon}.
In Figure~\ref{krigingModemonotone}, the case where the unconstrained Kriging mean and the inequality mode are different is considered. In both figures, the solid line represents the inequality mode $M_{\text{IK}}^N$ of the finite-dimensional approximation when $N=500$. The Gaussian covariance function is used with the length parameter $\theta=0.3$. In Figure~\ref{modemon500N20}, the dashed-line corresponds to the inequality mode when $N=20$, which is close to one generated from $N=500$.

\begin{figure}[hptb]
\begin{minipage}{.5\linewidth}
\centering
\subfloat[]{\label{krigingModemonotone}\includegraphics[scale=.3]{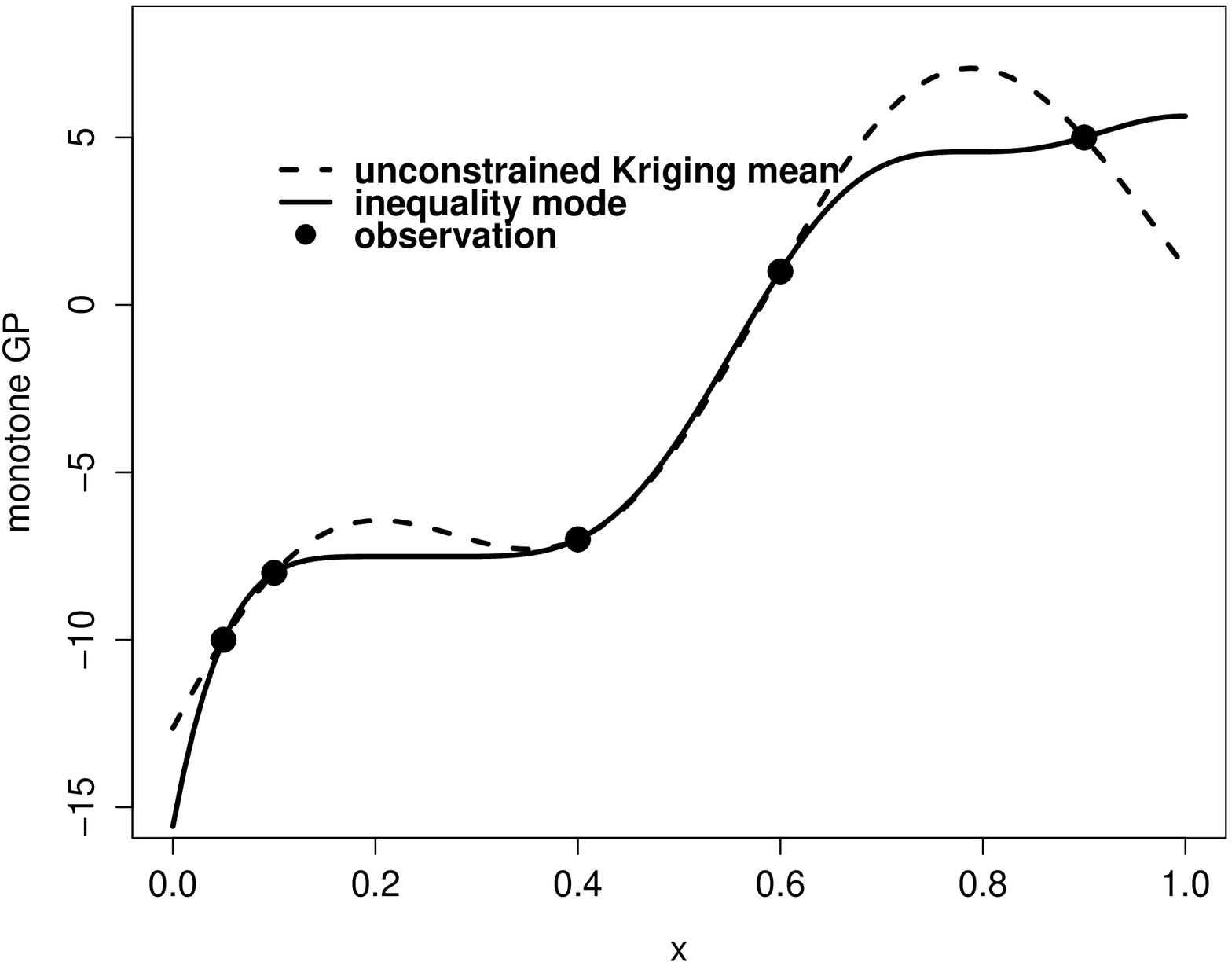}}
\end{minipage}%
\begin{minipage}{.5\linewidth}
\centering
\subfloat[]{\label{modemon500N20}\includegraphics[scale=.3]{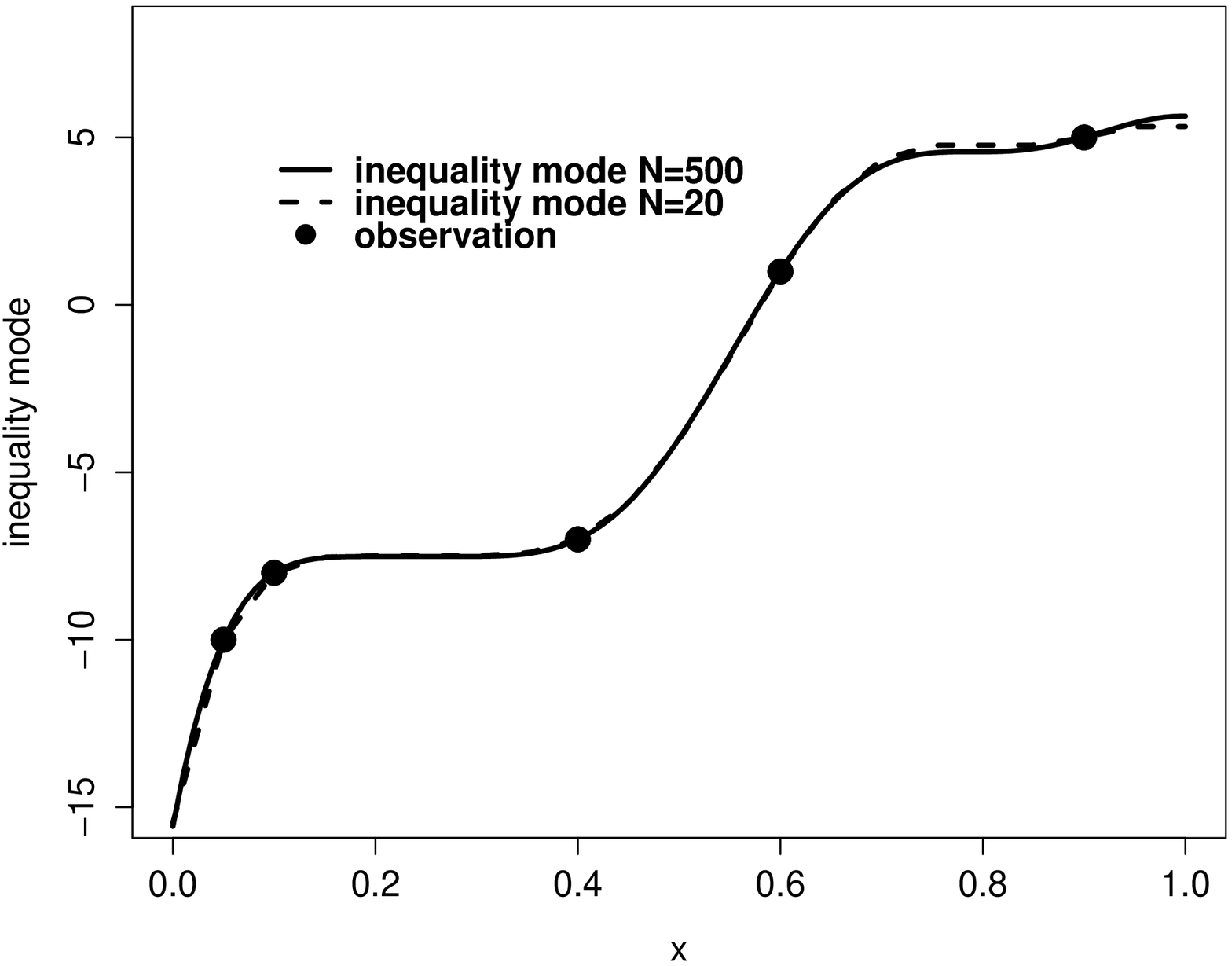}}
\end{minipage}
\caption{In both figures, the solid line represents the inequality mode $M^N_{\text{IK}}$ when $N=500$ which respects monotonicity constraints in the entire domain. The dashed-line in Figure~\ref{krigingModemonotone} (resp. Figure~\ref{modemon500N20}) represents the usual Kriging mean (resp. the inequality mode when $N=20$).}
\label{approachmodemon}
\end{figure}

\section{Conclusion} 
In this article, we propose a new model for incorporating both interpolation conditions and inequality constraints into a Gaussian process emulator. Our method ensures that the inequality constraints are respected not only in a discrete subset of the input set but also in the entire domain. We suggest a finite-dimensional approximation of Gaussian processes which converges uniformly pathwise. It is constructed by incorporating deterministic basis functions and Gaussian random coefficients. We show that the basis functions can be chosen such that inequality constraints of $Y^N$ are \textit{equivalent} to a finite number of constraints on the coefficients. So, the initial problem is \textit{equivalent} to simulate a Gaussian vector restricted to convex sets. This model has been applied to real data in assurance and finance to estimate a term-structure curve and default probabilities (see \cite{2016arXiv160402237C} for more details).\

Now, the problem is open to substantial future work. For practical applications, estimating parameters should be investigated and Cross Validation techniques can be used. The suited Cross Validation method to inequality constraints described in \cite{Maatouk201538} can be developed. As input dimension increases, the efficiency of the method will become low. In fact, the size of the Gaussian vector of random coefficients in the approximation model increases exponentially. However, the choice of knots (subdivision of the input set) can be improved to reduce the cost of simulation, as well as the number of basis functions. This problem is also related to the choice of the basis functions with respect to \textit{a prior} information on the regularity of the real function. Additionally, the simulation of the truncated Gaussian vector can be accelerated by Markov chain Monte Carlo (McMC) methods or Gibbs sampling (see e.g. \cite{Geweke91efficientsimulation} and \cite{Robert}).

\begin{acknowledgements}
Part of this work has been conducted within the frame of the ReDice Consortium, gathering industrial (CEA, EDF, IFPEN, IRSN, Renault) and academic (\'Ecole des Mines de Saint-\'Etienne, INRIA, and the University of Bern) partners around advanced methods for Computer Experiments. The authors also thank Olivier Roustant (ENSM-SE) and Yann Richet (IRSN) for helpful discussions.
\end{acknowledgements}

\bibliographystyle{spmpsci}      
\bibliography{biblio}

\end{document}